\documentclass[11pt]{amsart}
\usepackage[margin=1in]{geometry}
 \numberwithin{equation}{section}

\newcommand{\rt}{\frac{\rho^{1-\varepsilon}}{\rho^{\alpha-\varepsilon}+\tau}}
\newcommand{\rtp}{\frac{(\rho^+)^{1-\varepsilon}}{(\rho^+)^{\alpha-\varepsilon}+\tau}}
\newcommand{\rtk}{\frac{(\rho_k)^{1-\varepsilon}}{(\rho_k)^{\alpha-\varepsilon}+\tau}}
\newcommand{\rtj}{\frac{(\nj)^{1-\varepsilon}}{(\nj)^{\alpha-\varepsilon}+\tau}}

\newcommand{\mbd}{\textup{div}}
\newcommand{\sn}{\sqrt{u}}

\newcommand{\io}{\int_\Omega}

\newcommand{\wt}{W^{2,2}(\Omega)}

\newcommand{\nss}{{\nabla^2}}
\newcommand{\uat}{{ u^{\frac{\alpha}{2}}}}
\newcommand{\urt}{{ u^{\frac{\gamma}{2}}}}

\newcommand{\ua}{{ u^\alpha}}
\newcommand{\ur}{{ u^\gamma}}
\newcommand{\ue}{{ u^\eta}}

\newcommand{\nua}{{\nabla u^\alpha}}
\newcommand{\nur}{{\nabla u^\gamma}}
\newcommand{\ub}{{ u^\beta}}

\newcommand{\nurt}{{\nabla u^{\frac{\gamma}{2}}}}

\newcommand{\bg}{{\bf g}}

\newcommand{\po}{\partial\Omega}
\newcommand{\pot}{\Sigma_T}

\newcommand{\drk}{\Delta\rk}
\newcommand{\drka}{\Delta\rka}
\newcommand{\lno}{L^{\infty}(\Omega)}

\newcommand{\rnt}{(\rho+\tau)^n}
\newcommand{\rnk}{(\rho_k+\tau)^n}
\newcommand{\rnj}{(\nj+\tau)^n}
\newcommand{\rpnt}{\left(\rho^++\tau\right)^n}

\newcommand{\nja}{\nj^\alpha}
\newcommand{\rks}{\rho_k}
\newcommand{\rk}{\rho_k}
\newcommand{\nrk}{\nabla\rho_k}
		\newcommand{\nrks}{|\nabla\rho_k|^2}
		\newcommand{\kr}{K(\rho_k)}
		\newcommand{\mr}{M(\rho_k)}
		\newcommand{\lr}{L(\rho_k)}
\newcommand{\rka}{\rho_k^\alpha}
\newcommand{\rkb}{\rho_k^\beta}
\newcommand{\rkm}{\rho_k ^{\varepsilon-1}}
\newcommand{\rkam}{\rho_k ^{\alpha-1}}

\newcommand{\fk}{F_k}
\newcommand{\fj}{\overline{F}_j}

\newcommand{\nj}{\overline{u}_j}
\newcommand{\njt}{\tilde{u}_j}

\newtheorem{thm}{Theorem}[section]
\newtheorem{lem}{Lemma}[section]

\newtheorem{cor}{Corollary}[section]
\begin{document}
\title[Fourth Order Parabolic Equations]{A Class of Functional Inequalities and their Applications to Fourth-Order Nonlinear Parabolic Equations}
\author{Jian-Guo Liu and Xiangsheng Xu}\thanks
{ Liu's address: Department of Physics and Department of Mathematics, Duke University, Durham, NC 27708, {\it Email:} jliu@phy.duke.edu.\\
	Xu's address: Department of Mathematics and Statistics, Mississippi State
University, Mississippi State, MS 39762.
{\it Email}: xxu@math.msstate.edu. {\it Commun. Math. Sci.}, to appear.}
 \keywords{ Existence, Nonlinear fourth order
parabolic equations, Thin-film equation, Quantum drift-diffusion
model, Functional inequalities. } \subjclass{ 35D30, 35A01, 35K25}
\begin{abstract} We study a class of fourth order nonlinear parabolic equations which include the thin-film equation and the quantum drift-diffusion model as special cases. We investigate these equations by first developing functional inequalities of the type 
$$
\int_\Omega  u^{2\gamma-\alpha-\beta}\Delta u^\alpha\Delta u^\beta dx
\geq c\int_\Omega|\Delta u^\gamma
|^2dx,
$$
which seem to be of interest on their own right. 
\end{abstract}
\maketitle

\section{Introduction}\label{sec1}
Let $T>0$ and $\Omega$ be a domain in $\mathbb{R}^N$ with boundary
$\partial\Omega$. We consider the existence of a solution to the
problem
\begin{eqnarray}
\partial_tu+\mbd[u^n\nabla(u^{\alpha-1}\Delta\ua)]&=&0\
\ \ \mbox{in $\Omega_T$},\label{311} \\
\nabla u\cdot\nu=u^n\nabla(u^{\alpha-1}\Delta\ua)\cdot\nu &=&
0 \ \ \
\mbox{on $\Sigma_T$,}\\
u(x,0)&=&u_0(x)\geq 0 \ \ \ \mbox{on $\Omega$},\label{312}
\end{eqnarray}
where $\Omega_T= \Omega\times(0, T]$, $\Sigma_T=\partial\Omega\times(0,T]$, $\nu$ is the unit outward normal to $\po$. Numbers
$ n, \alpha\in (0,\infty)$ and  the functions $\bg=\bg(x,t), \ u_0(x)$ are
given data whose precise assumptions will be made later.

Fourth-order nonlinear parabolic equations arise in a variety of physical settings \cite{DFL,EG,GR,LX}. Two well-known examples are the thin film equation and the quantum drift-diffusion model, both of which are special cases of \eqref{311}. In a typical thin film equation, we have that $\alpha=1, n>0$, while parameter values of
$n=1, \alpha=\frac{1}{2}$ give us the quantum drift-diffusion equation without the drift term. See, e.g., \cite{GST,X4} for the inclusion of this term. Note that
the drift term is a lower order term, and dropping it simply implies that we have assumed that
it can be dominated by the principal term in the equation. 
Nonetheless, 
extensive research work has been done on these two
types of problems. We refer the reader to \cite{GST,J2,MMS,X2} and the references therein.

The objective of our work is to present a unified mathematical approach to these two very different physical problems. This is done via functional inequalities of the type
\begin{equation}\label{dsmm1}
I(u)\equiv\int_\Omega u^{2\gamma-\alpha-\beta}\Delta\ua\Delta\ub dx
\geq c\int_\Omega\left(\Delta\ur\right)^2dx\ \ \ \mbox{for all $u\in W_\gamma$},
\end{equation}
where
\begin{equation}\label{1.13}
W_\gamma= \{u\geq 0: \ur\in\wt, \nabla\ur\cdot\nu=0 \ \ 
\mbox{on $\po$}\}.
\end{equation}
Obviously, the validity of the above inequality  depends on $\Omega, \alpha, \beta$ and
$\gamma$. We will focus on the case where $\Omega$ is bounded and
convex. Then a result of \cite{G} asserts that
\begin{equation}\label{n231}
\int_\Omega(\Delta\ur)^2dx\geq\int_\Omega|\nss\ur|^2dx
\end{equation} 
for all $u\in W_\gamma$, where $\nss\ur$ denotes the Hessian of $\ur$. Thus a slightly weaker version is the inequality
\begin{equation}\label{1dsmm1}
I(u)\equiv\int_\Omega u^{2\gamma-\alpha-\beta}\Delta\ua\Delta\ub dx\geq c\int_\Omega\left(\nss\ur\right)^2dx\ \ \ \mbox{for all $u\in W_\gamma$}.
\end{equation}
Several known inequalities are special cases of this.
If $\beta=1, \alpha=\gamma=\frac{1}{2}$, then \eqref{1dsmm1} is established for box domains with sides parallel to the coordinate planes in \cite{CD} (also see \cite{JM}).
It turns out \cite{GST, MMS} that \eqref{1dsmm1} is still valid if $\beta=1, \gamma=\alpha\in\left(\frac{(N-1)^2}{2N^2+1}, \frac{3}{2}\right)$, and $\Omega$ is a bounded convex domain. The inequalities in \cite{GST, MMS} are formulated in a measure-theoretic setting. See \cite{X3} for a more direct
approach.

The significance of functional inequalities of the type \eqref{dsmm1} lies in the fact that the integrand on the left-hand side of \eqref{dsmm1} can change signs.
In essence, they are the nonlinear version of the G\aa rding inequality. To illustrate how they arise naturally in the study of fourth order nonlinear partial differential equations,
we proceed to make some
formal analysis of \eqref{311}-\eqref{312}. That is, we assume that $u$ is a positive, smooth solution of \eqref{311}.
Use $\ub$, where $\beta>0$, as a test function in \eqref{311} to derive
\begin{eqnarray}
\frac{1}{\beta+1}\frac{d}{dt}\int_\Omega u^{\beta+1}dx+\frac{\beta}{n+\beta}\int_\Omega u^{\alpha-1}
\Delta\ua\Delta u^{n+\beta}dx=0.\label{rma11}
\end{eqnarray}
By \eqref{dsmm1}, we have 
\begin{equation}\label{119}
\int_\Omega u^{\alpha-1}
\Delta\ua\Delta u^{n+\beta}dx\geq c\int_\Omega\left(\Delta u^{\frac{2\alpha+n+\beta-1}{2}}\right)^2dx.
\end{equation}
For the moment, we ignore the restrictions under which the above inequality holds. We will address this issue in Section \ref{sec2}. 
Integrate \eqref{rma11} to obtain
\begin{equation}
\max_{0\leq t\leq T}\int_\Omega u^{\beta+1}(x,t)dx+\int_{\Omega_T}\left(\Delta u^{\frac{2\alpha+n+\beta-1}{2}}\right)^2dxds\leq c.\label{111}
\end{equation}

Our study of \eqref{dsmm1} is inspired by the integration by parts rule proved by Gianazza et al. \cite{GST} and by J\"{u}ngel and Mattes \cite{JM}. We also refer the reader to \cite{JM2} for the development of an algebraic technique for dealing with such formulas.
The framework we have developed here is also algebraic in nature, but it seems to be more direct and easier to use. This can best be illustrated by the application of our method to the standard thin film 
\begin{equation}
\partial_tu+\textup{div}\left(u^n\nabla\Delta u\right)=0.
\end{equation}
In this case, the second integral in \eqref{rma11} becomes
$$\int_{\Omega}\Delta u\Delta u^{n+\beta}dx.$$
This immediately puts us in a position to apply Lemma \ref{l2.5} in Section \ref{sec2}, from whence follows that for each
$\beta\in (\frac{1}{2}-n, 2-n)$ there is a positive number $c$ such that
$$ \int_{\Omega}\Delta u\Delta u^{n+\beta}dx\geq c \int_{\Omega}\left(\Delta u^{\frac{n+\beta+1}{2}}\right)^2dx.$$
Of course, this result is well-known, see, e.g., \cite{JM2} and the references therein. 
Also notice how easy it is for us to prove Lemma \ref{l2.5} in our framework. More importantly,
our method has led to the discovery of Corollary \ref{cor2.2} in Section \ref{sec2}. It is this corollary that enables us to  solve a problem left open  in \cite{MMS}. 

We can easily foresee other potential applications for the functional inequalities developed in this paper. An immediate example
is the study of epitaxial growth of thin films ( see \cite{AKW,GLL}) and the references therein). A family of continuum models has been established, one of which has the form
\begin{equation}
\partial_t u+u^2\Delta^2 u^3=0\ \ \mbox{in}\ \Omega_T.\label{rat11}
\end{equation} 
Using $\ub$ as a test function yields
\begin{equation}
\frac{1}{\beta+1}\frac{d}{dt}\int_{\Omega} u^{\beta+1}dx+\int_{\Omega}\Delta u^3\Delta u^{\beta+2}dx=0,
\end{equation}
and Lemma \ref{l2.5} in Section \ref{sec2} becomes applicable. Of course, the resulting inequality is far from enough to obtain an existence assertion for \eqref{rat11}. However,  the idea behind the derivation of the inequality  can lead to the discovery of additional estimates. 
Since our inequalities do not depend on the space dimension $N$, their applications will inevitably lead to the relaxation of the restrictions on $N$ in previous studies such as \cite{GLL}.   

\begin{thm}\label{thm1.1} Let $\Omega$ be a bounded convex domain in $\mathbb{R}^N$. Assume:
	\begin{enumerate}
		\item[(H1)] $\alpha\in [1,\frac{3}{2}), n\in[1, 1+\frac{\sigma}{4})$,
		where
		\begin{equation}\label{1115}
		\sigma=\left\{\begin{array}{ll}
		1&\mbox{if $N<4$,}\\
		\frac{4}{N}&\mbox{if $N>4$,}\\
		\mbox{any number in $(0,1)$}&\mbox{if $N=4$;}
		\end{array}\right.
		\end{equation}
		\item[(H2)] $u_0\in L^\infty(\Omega)$ with $\inf_{\Omega}u_0>0$.
	\end{enumerate}
	Then there is a weak solution to \eqref{311}-\eqref{312} in the following sense:
	\begin{enumerate}
		\item[(C1)] $u\in L^{2\alpha+\sigma}(\Omega_T)$ with $u\geq 0$ on $\Omega_T$, $\ua
		\in L^2(0,T; W^{2,2}(\Omega))$;
		\item[(C2)] $\nabla\ua\cdot\nu=0$ a.e. on $\Sigma_T$;
		\item[(C3)] for each $\xi\in C^\infty(\overline{\Omega_T})$ with $\xi(x,T)=0$ and $\nabla\xi\cdot\nu=0$ on $\Sigma_T$ there holds
		\begin{eqnarray}
		\lefteqn{-\int_{\Omega_T} u\partial_t\xi dxdt-\io u_0(x)\xi(x,0)dx}\nonumber\\
		&&+\int_{\Omega_T}\left(\frac{2n}{\alpha} u^{n+\frac{\alpha}{2}-1}\nabla\uat\Delta\ua\nabla\xi+ u^{\alpha+n-1}\Delta\ua\Delta\xi\right)dxdt=0.\label{116}
		\end{eqnarray}
	\end{enumerate}
\end{thm}
We would like to make some remarks about Theorem \ref{thm1.1}. 
We can conclude
from Lemma \ref{l22} below that $\nabla\uat\in \left(L^4(\Omega_T)\right)^N$. Thus each integral in \eqref{116} makes sense. Assumption (H1) is largely due to the restrictions for \eqref{dsmm1} to hold.
\begin{thm}\label{thm1.2} Let $\Omega$ be a bounded convex domain in $\mathbb{R}^N$ and (H2) hold. Assume:
	\begin{enumerate}
		\item[(H3)] $\alpha=1, n\in(\frac{1}{2}, 1+\frac{\sigma}{4})$, where $\sigma$ is given as in \eqref{1115}.
	\end{enumerate}
	Then there is a weak solution to \eqref{311}-\eqref{312} in the sense of (C3).
\end{thm}

In comparison with previous results on the thin-film equation (see, e.g., \cite{CPT,DGG1,GR,GR2}), this theorem has removed all the restrictions on space dimensions.
Thus this is truly a multi-dimensional result. The trade-off is that our assumption on $n$ in the theorem is weaker than those in \cite{DGG1,GR2}. It is worth noting that most of the existing results on non-linear fourth-order parabolic equations involve restrictions on the space dimensions with
the one-dimensional problems attracting the most attention. See ,e.g., (\cite{BF,BP,DFL,DGG,WBB}), where various properties of solutions are investigated. More recent results of this nature on the thin-film equation can be found in \cite{FG,GGKO,GP}.

Our approach to the question of existence is to construct a sequence of smooth, positive approximate solutions such that the calculations similar to \eqref{rma11}-\eqref{111} can be employed. A well-known difficulty in the study of fourth-order equations is that the maximum principle is no longer true. In fact, the heat kernel for the heat biharmonic equation changes signs. Thus arguments based upon the maximum principle for second order equations do not work here. We must rely on the nonlinear structure of our equation to obtain non-negative solutions. It turns out that the term $u^{\alpha-1}=\frac{1}{u^{1-\alpha}}$ in \eqref{311} plays a key role in the existence of non-negative solutions. The case where $n=1, \alpha\leq 1$ has already been considered in \cite{MMS,X4}, while the case where $\alpha>1$ is left open there. One contribution of this paper is that we have completely solved this open problem (Theorem 1.1). Even though we have not been able to find a physical application for this case, it is still very interesting from the point of view of mathematical analysis because this is the case where the gradient flow theory fails \cite{MMS}.  The key to our success seems to be that we have found a right way to approximate the term $u^{1-\alpha}$ with the exponent being negative.

The optimal transport theory has been successfully employed to treat
many different types of parabolic equations as gradient flows of various “entropy
functionals” for various “transportation metrics”, the canonical example being the
regular scalar heat equation viewed by Jordan, Kinderlehrer and Otto \cite{JKO} as the
gradient flow of the Boltzmann entropy for the quadratic Monge-Kantorovich MK2
(frequently named Wasserstein) metric. We have seen a very large body of work done on this
subject in the last 20 years ( in the study of the heat equation in a very general framework,
porous-medium equations, thin-film flow equations, chemotaxis models, etc.. See \cite{GST,MMS,LMS} and the references therein
as examples.).
However, in the generality considered in Theorems 1.1 and 1.2, the
transport theory is no longer applicable \cite{MMS}. We discretize the time derivative in \eqref{311} and transform it into a system of two second order elliptic equations. Our approximation scheme seems to be standard. However, the genius is in the details, and we have to overcome
numerous technical difficulties for it to work here. On the one hand, we need to introduce new terms in our
approximate problems in order to ensure high regularity and positivity of our approximate solutions. On the other hand, we have to make sure that these new terms do not destroy the essential a prior estimates that hold for positive, smooth solutions of the original equations.  Striking a suitable balance between the two constitutes the core of our development. 

This paper is organized as follows. In section \ref{sec2} we develop a class of functional inequalities. Section \ref{sec3} is devoted to the fabrication of our approximation schemes. Here the key is how to handle the term $u^{\alpha-1}$. Then we proceed to
obtain discretized versions of the a priori estimates that hold for positive, smooth solutions of the original equations, which eventually leads to the establishment of Theorems 1.1 and 1.2 in the two subsequent sections.

\section{Functional Inequalities}\label{sec2}
\setcounter{equation}{0}

In this section we study the functional inequality \eqref{dsmm1}. We will focus on the case where $\Omega$ is
a bounded convex domain in $\mathbb{R}^N$. Our method is algebraic in nature. In this regard, it is similar to \cite{JM2}.

The key to our development is the following lemma, which is a substantial improvement over Lemma 2.1 in \cite{X4}.
\begin{lem}\label{l21}Let  $\Omega$ be a bounded domain in $\mathbb{R}^N$ with Lipschitz
	boundary $\po$. Assume that
	\begin{equation}
	\alpha\ne 0.
	\end{equation} 
	Then we have
	\begin{eqnarray}
	\lefteqn{		\int_\Omega u^{2\alpha-2\beta}|\nss u^\beta|^2
		dx\geq \frac{2\beta^2}{(2+N)\alpha^2}\int_\Omega|\nabla^2\ua|^2dx}\nonumber\\
	&&+\frac{\beta^2}{(2+N)\alpha^2}\int_\Omega(\Delta\ua)^2dx
	+\frac{16\beta^2(\alpha-\beta)(\alpha-3\beta)}{(2+N)\alpha^4}\int_\Omega|\nabla\uat|^4dx\label{5.2}
	\end{eqnarray}
	for all $u\in W_\alpha$.
\end{lem}

\begin{proof}If $\beta=0$, then the lemma is trivially true. Thus assume
	that $\beta\ne 0$. Note that
	\begin{equation}\label{n2.3}(\Delta\ub)^2\leq N|\nss\ub|^2.
	\end{equation}
	Thus if $\alpha=\beta$, then \eqref{5.2} is still true. From here on, we let $$
	\beta\ne \alpha.$$
	We can also assume that $u\in  W_\alpha$ is 
	bounded away from $0$ below. If this is not the case, we can always replace $u$ by
	\begin{equation*}
	\left(u^\alpha+\varepsilon\right)^{\frac{1}{\alpha}}
	\end{equation*}
	and then let $\varepsilon\rightarrow 0^+$.
	The same is
	understood in the subsequent calculations in this section.
	We compute, for $i,j=1, \cdots,
	N $, that
	\begin{eqnarray}
	\partial_i\ub&=&\partial_i(\ua)^{\frac{\beta}{\alpha}}=\frac{\beta}{\alpha}u^{\beta-\alpha}\partial_i\ua,\\
	\partial^2_{ij}\ub &=&
	\frac{\beta(\beta-\alpha)}{\alpha^2}u^{\beta-2\alpha}\partial_i\ua\partial_j\ua+\frac{\beta}{\alpha}u^{\beta-\alpha}\partial^2_{ij}\ua.\label{2.3}
	\end{eqnarray}
	First, we let $i=j$ in the above equation and then sum up over $i$ to derive
	\begin{equation}
	\Delta\ub=\frac{\beta(\beta-\alpha)}{\alpha^2}u^{\beta-2\alpha}|\nabla\ua|^2+\frac{\beta}{\alpha}u^{\beta-\alpha}\Delta\ua.\label{2.4}
	\end{equation}
	Square both sides  of this equation and multiply through the resulting
	equation by $\frac{\alpha^2}{\beta^2}u^{2\alpha-2\beta}$ to arrive at
	\begin{equation}
	\frac{\alpha^2}{\beta^2}u^{2\alpha-2\beta}|\Delta\ub
	|^2=|\Delta\ua|^2+2\frac{\beta-\alpha}{\alpha}\frac{1}{\ua}|\nabla\ua|^2
	\Delta\ua+\left(\frac{\beta-\alpha}{\alpha}\right)^2\frac{1}{u^{2\alpha}}|\nabla\ua|^4.\label{22.5}
	\end{equation}
	Square both sides of \eqref{2.3}, multiply through the resulting
	equation by $\frac{\alpha^2}{\beta^2}u^{2\alpha-2\beta}$, and then
	sum up $i, j$ to obtain
	\begin{equation}
	\frac{\alpha^2}{\beta^2}u^{2\alpha-2\beta}|\nabla^2\ub
	|^2=|\nabla^2\ua|^2+2\frac{\beta-\alpha}{\alpha}\frac{1}{\ua}\nabla\ua\cdot
	\nabla^2\ua\nabla\ua+\left(\frac{\beta-\alpha}{\alpha}\right)^2\frac{1}{u^{2\alpha}}|\nabla\ua|^4.\label{2.5}
	\end{equation}
	Note that $\nabla\ua=2\uat\nabla\uat$. Keeping this in mind, we can rewrite \eqref{2.5} and \eqref{22.5} as
	\begin{eqnarray}
	2\nabla\uat\cdot
	\nabla^2\ua\nabla\uat&=&
	\frac{\alpha^3}{4(\beta-\alpha)\beta^2}u^{2\alpha-2\beta}|\nabla^2\ub
	|^2-\frac{\alpha}{4(\beta-\alpha)}|\nabla^2\ua|^2\nonumber\\	& &-\frac{4(\beta
		-\alpha)}{\alpha}|\nabla\uat|^4,\label{22.6}\\
	|\nabla\uat|^2
	\Delta\ua&=&
	\frac{\alpha^3}{8(\beta-\alpha)\beta^2}u^{2\alpha-2\beta}|\Delta\ub
	|^2-\frac{\alpha}{8(\beta-\alpha)}|\Delta\ua|^2\nonumber\\& &-\frac{2(\beta
		-\alpha)}{\alpha}|\nabla\uat|^4.\label{22.7}
	\end{eqnarray}
	Note that 
	\begin{eqnarray}
	u^{-2\alpha}|\nua|^4&=&u^{-2\alpha}|\nua|^2\nua\cdot\nabla\ua \nonumber\\
	&=&\textup{div}\left( u^{-2\alpha}|\nua|^2\nua\ua\right)-\textup{div}\left( u^{-2\alpha}|\nua|^2\nua\right)\ua \nonumber\\
	&=&\textup{div}\left( u^{-2\alpha}|\nua|^2\nua\ua\right)\nonumber\\
	&&-u^{-\alpha}|\nua|^2\Delta\ua-2u^{-\alpha}\nabla^2\ua\nabla\ua\cdot\nua+2u^{-2\alpha}|\nua|^4.\label{22.8}
	\end{eqnarray}
	Integrating this equation over $\Omega $, we obtain, with the aid of the fact that  $\nua\cdot\nu=0$ on $\partial\Omega$, that
	\begin{equation}\label{22.9}
	4\int_\Omega|\nabla u^{\frac{\alpha}{2}}|^4dx=2\int_\Omega\nabla
	u^{\frac{\alpha}{2}}\cdot(\nss\ua\nabla
	u^{\frac{\alpha}{2}})dx+\int_\Omega|\nabla u^{\frac{\alpha}{2}}|^2\Delta\ua dx.
	\end{equation}
	Integrate \eqref{22.6} and \eqref{22.7}  over $\Omega$, add the two resulting equations, then make use of
	\eqref{22.9}, thereby derive
	\begin{eqnarray}
	\lefteqn{\frac{\alpha^3}{4(\beta-\alpha)\beta^2}\int_\Omega u^{2\alpha-2\beta}|\nabla^2\ub
		|^2dx+\frac{\alpha^3}{8(\beta-\alpha)\beta^2}\int_\Omega u^{2\alpha-2\beta}|\Delta\ub
		|^2dx}\nonumber\\
	&=& \frac{\alpha}{4(\beta-\alpha)}\int_\Omega|\nabla^2\ua|^2dx
	+\frac{\alpha}{8(\beta-\alpha)}\int_\Omega|\Delta\ua|^2dx\nonumber\\
	&&-\frac{2(\alpha-3\beta)}{\alpha}\int_\Omega|\nabla\uat|^4dx.
	\end{eqnarray}
	Multiplying through this equation by $\frac{4(\beta-\alpha)\beta^2}{\alpha^3}$, we 
	can conclude the lemma from the inequality \eqref{n2.3}.
	The proof is complete.
\end{proof}

Notice that the only inequality we have used in the proof of the above lemma is \eqref{n2.3}. Thus \eqref{5.2} is just as sharp an inequality as \eqref{n2.3}. Obviously, the lemma has been obtained by sharpening
the proof of Lemma 2.1 in \cite{X4}.
\begin{lem}\label{l22}
	Assume that $\Omega$ is bounded and convex. Then we have
	\begin{equation}\label{dsm5}
	\int_\Omega|\nabla u^{\frac{\alpha}{2}}|^4 dx\leq
	\frac{9}{16}\int_\Omega(\Delta\ua)^2dx
	\end{equation}
	for all $u\in W_\alpha$.
\end{lem}
\begin{proof} This lemma is taken from \cite{X4}. The proof is rather simple. Thus we repeat it here.
	
	Remember that in this case \eqref{n231} holds. Taking note of this,
	we calculate from \eqref{22.9} that
	\begin{eqnarray}
	4\int_\Omega|\nabla
	u^{\frac{\alpha}{2}}|^4dx&\leq&2\left(\int_\Omega|\nss\ua|^2dx\right)^{\frac{1}{2}}\left(\int_\Omega|\nabla
	u^{\frac{\alpha}{2}}|^4dx\right)^{\frac{1}{2}}\nonumber\\
	&&+\left(\int_\Omega|\nabla
	u^{\frac{\alpha}{2}}|^4x\right)^{\frac{1}{2}}\left(\int_\Omega|\Delta\ua|^2dx\right)^{\frac{1}{2}}\label{n124}\\
	&\leq&
	3\left(\int_\Omega|\Delta\ua|^2dx\right)^{\frac{1}{2}}\left(\int_\Omega|\nabla
	u^{\frac{\alpha}{2}}|^4dx\right)^{\frac{1}{2}}\nonumber
	\end{eqnarray}
	from whence the lemma follows.
\end{proof}

Now we are ready to study the functional
\begin{equation}
I(u)=\int_\Omega u^{2\gamma-\alpha-\beta}\Delta\ua\Delta\ub dx.\label{1dtm111}
\end{equation}
At this point, we only assume
\begin{equation}\label{1.19}
\alpha\beta>0, \ \ \ \gamma\ne 0.\ \ \  
\end{equation}
Recall from \eqref{2.4} that
\begin{eqnarray}
\Delta\ua&=&\frac{\alpha(\alpha-\gamma)}{\gamma^2}u^{\alpha-2\gamma}|\nabla\ur|^2+\frac{\alpha}{\gamma}u^{\alpha-\gamma}\Delta\ur,\label{22.19}\\
\Delta\ub&=&\frac{\beta(\beta-\gamma)}{\gamma^2}u^{\beta-2\gamma}|\nabla\ur|^2+\frac{\beta}{\gamma}u^{\beta-\gamma}\Delta\ur.
\end{eqnarray}
Plugging these two into \eqref{1dtm111} yields
\begin{eqnarray}
\frac{\gamma^2}{\alpha\beta}I(u)&=& \int_\Omega(\Delta\ur)^2dx+\frac{16(\alpha-\gamma)(\beta-\gamma)}{\gamma^2}\int_\Omega|\nurt|^4dx\nonumber\\
&&+\frac{4(\alpha+\beta-2\gamma)}{\gamma}\int_\Omega|\nurt|^2\Delta\ur dx.\label{22.21}
\end{eqnarray}

Let us first consider the special case where $N=1$. In this case, we have
$$\int_\Omega\nurt\nabla^2\ur\nurt  dx=\int_\Omega |\nurt|^2\Delta\ur dx.$$
Thus by \eqref{22.9}, we obtain
$$\int_\Omega |\nurt|^2\Delta\ur dx=\frac{4}{3}\int_\Omega|\nurt|^4dx.$$ 
Use this in \eqref{22.21} to derive
\begin{equation}
\frac{\gamma^2}{\alpha\beta}I(u)=
\int_\Omega(\Delta\ur)^2dx+\frac{16(\gamma^2-2(\alpha+\beta)\gamma+3\alpha\beta)}{3\gamma^2}\int_\Omega|\nurt|^4dx.\label{dsm11}
\end{equation}
If $\gamma^2-2(\alpha+\beta)\gamma+3\alpha\beta\geq 0$, we are done. If 
$\gamma^2-2(\alpha+\beta)\gamma+3\alpha\beta< 0$,
i.e.,
\begin{equation}
\alpha+\beta-\sqrt{\alpha^2+\beta^2-\alpha\beta}<\gamma<\alpha+\beta+\sqrt{\alpha^2+\beta^2-\alpha\beta},
\end{equation} then we apply \eqref{dsm5} to \eqref{dsm11} to get
\begin{equation}
\frac{\gamma^2}{\alpha\beta}I(u)\geq \frac{4\gamma^2-6(\alpha+\beta)\gamma+9\alpha\beta}{\gamma^2}\int_\Omega(\Delta\ur)^2dx.\label{dsm2}
\end{equation}
For the coefficient of the integral in the preceding inequality to be positive, we must impose the conditions
\begin{eqnarray}
\gamma>\frac{3}{2}\alpha \ \ \ \mbox{or}\ \ \ \gamma&<&\frac{3}{2}\beta \ \ \ \mbox{in the case where $\alpha\geq\beta$, or} \\
\gamma>\frac{3}{2}\beta \ \ \ \mbox{or}\ \ \ \gamma&<&\frac{3}{2}\alpha \ \ \ \mbox{in the case where $\alpha<\beta$.}
\end{eqnarray}
In summary, we have
\begin{lem}\label{l2.3}
	If $N=1$ and $\alpha\geq\beta$, then \eqref{dsmm1} holds whenever
	\begin{eqnarray}
	\gamma&>&\min\left\{ \frac{3}{2}\alpha, \alpha+\beta+\sqrt{\alpha^2+\beta^2-\alpha\beta},\right\}\ \ \ \mbox{or}\\
	\gamma&<&\max\left\{ \frac{3}{2}\beta, \alpha+\beta-\sqrt{\alpha^2+\beta^2-\alpha\beta}\right\}.
	\end{eqnarray}
\end{lem}
Now we deal with the more general case $N>1$. It turns out that the sign of the term $2\gamma-\alpha-\beta$ plays a significant role.
\begin{lem}\label{l2.4}
	Let $\Omega$ be a bounded convex domain in $\mathbb{R}^N$ and $\gamma$ a number satisfying 
	\begin{equation}
	2\gamma-\alpha-\beta>0.\label{222}
	\end{equation} Without loss of any generality, we assume
	\begin{equation}
	\beta\leq\alpha.\label{223}
	\end{equation}
	If either
	\begin{eqnarray}
	\frac{3}{5}(\alpha+\beta)>\gamma&\geq&\alpha,\ \ \ \mbox{ or}\label{224}\\ 
	\gamma&<&\min\left\{\alpha, \frac{3}{2}\beta\right\},\label{225}
	\end{eqnarray}
	then there is a positive number $c$ such that \eqref{dsmm1} holds.
\end{lem}
\begin{proof}
	Under \eqref{222}-\eqref{224}, the coefficient of the second integral in \eqref{22.21}
	is non-negative, while the coefficient of the third integral is negative. Thus we can deduce from \eqref{22.21} and \eqref{dsm5} that
	\begin{eqnarray}
	\frac{\gamma^2}{\alpha\beta}I(u)&\geq& \int_\Omega(\Delta\ur)^2dx
	+\frac{4(\alpha+\beta-2\gamma)}{\gamma}\int_\Omega|\nurt|^2\Delta\ur dx\nonumber\\
	&\geq& \int_\Omega(\Delta\ur)^2dx
	+\frac{4(\alpha+\beta-2\gamma)}{\gamma}\left(\int_\Omega|\nurt|^4dx\right)^{\frac{1}{2}}\left(\int_\Omega(\Delta\ur )^2dx\right)^{\frac{1}{2}}\nonumber\\
	&\geq& \int_\Omega(\Delta\ur)^2dx
	+\frac{3(\alpha+\beta-2\gamma)}{\gamma}\int_\Omega(\Delta\ur )^2dx\nonumber\\
	&=&\frac{3(\alpha+\beta)-5\gamma}{\gamma}\int_\Omega(\Delta\ur)^2dx.
	\end{eqnarray}
	The coefficient of the last integral in the above inequality is positive by \eqref{224}. This completes the proof of the first part of the lemma. 
	
	If $\gamma<\alpha$, then the coefficient of the second integral in \eqref{22.21} is negative. Then it follows from \eqref{22.21} and \eqref{dsm5} that 
	\begin{eqnarray}
	\frac{\gamma^2}{\alpha\beta}I(u)&\geq& \int_\Omega(\Delta\ur)^2dx+\frac{9(\alpha-\gamma)(\beta-\gamma)}{\gamma^2}\int_\Omega(\Delta\ur)^2dx\nonumber\\
	&&+\frac{3(\alpha+\beta-2\gamma)}{\gamma}\int_\Omega(\Delta\ur)^2dx\nonumber\\
	&=&\frac{4\gamma^2-6(\alpha+\beta)\gamma+9\alpha\beta}{\gamma^2}\int_\Omega(\Delta\ur)^2dx.
	\end{eqnarray}
	Note that $4\gamma^2-6(\alpha+\beta)\gamma+9\alpha\beta=4(\gamma-\frac{3}{2}\beta)(\gamma-\frac{3}{2}\alpha)$. Thus it is positive if \eqref{225} holds. The proof is complete.
\end{proof}
Next we analyze the case where $\gamma=\frac{\alpha+\beta}{2}$. In this direction, we have the following result.
\begin{lem}\label{l2.5}Let $\Omega$ be a bounded convex domain in $\mathbb{R}^N$.
	Then for each $\alpha\in(\frac{\beta}{2}, 2\beta)$ there is a positive number $c=c(\alpha, \beta)$ such that
	\begin{equation}
	\int_\Omega\Delta\ua\Delta\ub dx\geq c\int_\Omega\left(\Delta u^{\frac{\alpha+\beta}{2}}\right)^2dx
	\end{equation}
	for all $u\in W_{\frac{\alpha+\beta}{2}}$.
\end{lem} 
\begin{proof}. Let $\gamma=\frac{\alpha+\beta}{2}$ in \eqref{22.21} to obtain
	\begin{equation}
	\frac{(\alpha+\beta)^2}{4\alpha\beta}I(u)= \int_\Omega(\Delta u^{\frac{\alpha+\beta}{2}})^2dx-\frac{16(\alpha-\beta)^2}{(\alpha+\beta)^2}\int_\Omega|\nabla u^{\frac{\alpha+\beta}{4}}|^4dx
	\end{equation}
	In view of \eqref{dsm5}, we have
	\begin{equation}
	\frac{(\alpha+\beta)^2}{4\alpha\beta}I(u)\geq\left(1 -\frac{9(\alpha-\beta)^2}{(\alpha+\beta)^2}\right)\int_\Omega(\Delta u^{\frac{\alpha+\beta}{2}})^2dx.
	\end{equation}
	If $\alpha\in(\frac{\beta}{2}, 2\beta)$, then the coefficient on the right-hand side of the preceding inequality is positive. The proof is complete.
\end{proof}

For the case where 
\begin{equation}\label{dsm111}
2\gamma-\alpha-\beta<0,
\end{equation}  we deduce from \eqref{22.9} that
\begin{eqnarray}
\frac{\gamma^2}{\alpha\beta}I(u)&=& \int_\Omega(\Delta\ur)^2dx+\frac{16(\alpha\beta-\gamma^2)}{\gamma^2}\int_\Omega|\nurt|^4dx\nonumber\\
&&-\frac{8(\alpha+\beta-2\gamma)}{\gamma}\int_\Omega\nurt\nabla^2\ur\nurt  dx.\label{dtm3}
\end{eqnarray}
Hence the key is how to handle the term $\int_\Omega\nurt\nabla^2\ur\nurt  dx$. 
To this end, we infer from \eqref{22.6} that 
\begin{eqnarray}
\nabla\urt\cdot
\nabla^2\ur\nabla\urt&=&
\frac{\gamma^3}{8(\eta-\gamma)\eta^2}u^{2\gamma-2\eta}|\nabla^2\ue
|^2-\frac{\gamma}{8(\eta-\gamma)}|\nabla^2\ur|^2\nonumber\\	& &-\frac{2(\eta
	-\gamma)}{\gamma}|\nabla\urt|^4,\label{dtm2}
\end{eqnarray}
where $\eta$ is a number to be determined later.
Substituting this into \eqref{dtm3} , we arrive at
\begin{eqnarray}
\frac{\gamma^2}{\alpha\beta}I(u)&=& \int_\Omega(\Delta\ur)^2dx+\frac{\alpha+\beta-2\gamma}{\eta-\gamma}\int_\Omega|\nabla^2\ur|^2dx\nonumber\\
&&+\frac{16[(\alpha+\beta-2\gamma)(\eta-\gamma)+\alpha\beta-\gamma^2]}{\gamma^2}\int_\Omega|\nurt|^4dx\nonumber\\
&&+\frac{(\alpha+\beta-2\gamma)\gamma^2}{(\gamma-\eta)\eta^2}\int_\Omega u^{2\gamma-2\eta}|\nabla^2\ue|^2 dx.\label{dtm4}
\end{eqnarray}
This puts us in position to apply \eqref{5.2}. To do this, we need to suppose
\begin{equation}\label{df2}
\gamma-\eta>0
\end{equation}
to ensure the coefficient of the last integral in \eqref{dtm4} is positive. In our context, the inequality \eqref{5.2} has the form
\begin{eqnarray}
\lefteqn{		\int_\Omega u^{2\gamma-2\eta}|\nss u^\eta|^2
	dx\geq \frac{2\eta^2}{(2+N)\gamma^2}\int_\Omega|\nabla^2\ur|^2dx}\nonumber\\
&&+\frac{\eta^2}{(2+N)\gamma^2}\int_\Omega(\Delta\ur)^2dx
+\frac{16\eta^2(\gamma-\eta)(\gamma-3\eta)}{(2+N)\gamma^4}\int_\Omega|\nabla\urt|^4dx.
\end{eqnarray}
Use this in \eqref{dtm4} to derive
\begin{eqnarray}
\lefteqn{\frac{\gamma^2}{\alpha\beta}I(u)}\nonumber\\
&\geq& \frac{\alpha+\beta+N\gamma-(2+N)\eta}{(2+N)(\gamma-\eta)}\int_\Omega(\Delta\ur)^2dx
-\frac{(\alpha+\beta-2\gamma)N}{(\gamma-\eta)(N+2)}\int_\Omega|\nabla^2\ur|^2dx\nonumber\\
&&+\frac{16[((N-1)\eta-(N+1)\gamma)(\alpha+\beta-2\gamma)+(2+N)(\alpha\beta-\gamma^2)]}{(2+N)\gamma^2}\int_\Omega|\nurt|^4dx. \label{dtm5}
\end{eqnarray}
We choose $\eta$ so that the coefficient of the last integral in the above equation  is $0$. This leads to
\begin{equation}\label{22.38}
\eta=\frac{(N+1)\gamma(\alpha+\beta-2\gamma)-(2+N)(\alpha\beta-\gamma^2)}{(N-1)(\alpha+\beta-2\gamma)}.
\end{equation}
The number $\eta$ chosen above must satisfy \eqref{df2}. 
Plug the value of $\eta$ into \eqref{dtm5} and take a note of \eqref{n231} and the fact that the coefficient of the second integral in \eqref{dtm5} is negative
to arrive at
\begin{eqnarray*}
	\lefteqn{\frac{(2+N)\gamma^2}{\alpha\beta}I(u)}\nonumber\\
	&&\geq\frac{(1-N)(\alpha+\beta)+3N\gamma-(2+N)\eta}{\gamma-\eta}\int_\Omega(\Delta\ur)^2dx. \label{dsm1}
\end{eqnarray*}	
Thus our last hypothesis is that the coefficient of the above integral is positive, i.e.,
\begin{equation}\label{dsa3}
(1-N)(\alpha+\beta)+3N\gamma-(2+N)\eta>0.
\end{equation}
To summarize our results, we have
\begin{lem}\label{l2.6}
	Let $\Omega$ be a bounded convex domain in $\mathbb{R}^N$. Assume that \eqref{1.19}
	and \eqref{dsm111} hold. If $\eta$ given by \eqref{22.38} satisfies \eqref{df2}, and \eqref{dsa3},
	then there is a positive number $c=c(\alpha,\beta,\gamma, N)$ such that \eqref{dsmm1} holds.
\end{lem}

\begin{cor}\label{cor2.1}
	Let $\Omega$ be a  bounded convex domain in $\mathbb{R}^N$. Then for each $\alpha\in (\frac{(N-1)^2}{2N^2+1}, \frac{3}{2})$ there is a positive number $c=c(\alpha, N)$ such that
	\begin{equation}\label{dtm1}
	\int_\Omega u^{\alpha-1}\Delta u\Delta \ua dx\geq c\int_\Omega (\Delta \ua)^2dx
	\end{equation}
	for all $u\in W_\alpha$.
	
\end{cor}


\begin{proof} This corollary is largely contained in \cite{X4}. A different version can be found in \cite{MMS}. It is also an easy
	consequence of our preceding development. To see this, note that in this case we have
	\begin{equation}\label{dsa1}
	\beta=1,\ \ \ \gamma=\alpha,\ \ \ \mbox{and $2\gamma-\alpha-\beta=\alpha-1$}.
	\end{equation}
	If $\alpha=1$, then \eqref{dtm1} is trivially true. 
	If $\alpha>1$, we apply Lemma \ref{l2.4}. The conditions \eqref{222},  \eqref{223}, and \eqref{224} are equivalent to
	$$1<\alpha<\frac{3}{2}.$$
	If  $\alpha< 1$, we substitute \eqref{dsa1} into \eqref{22.38} to obtain
	\begin{equation}\label{22.42}
	\eta=-\frac{\alpha}{N-1}.
	\end{equation}
	Obviously, \eqref{1.19} is true. Since $\eta<0$, we see that \eqref{df2} is also satisfied. Plugging \eqref{dsa1} and \eqref{22.42} into \eqref{dsa3}, we arrive at
	\begin{equation}
	(2N^2+1)\alpha>(N-1)^2.
	\end{equation}
	Thus \eqref{dsa3} holds under our assumptions on $\alpha$. We conclude \eqref{dtm1} from Lemma \ref{l2.6}.	
	

\end{proof}
\begin{cor}\label{cor2.2}
	Let $\Omega$ be a  bounded convex domain in $\mathbb{R}^N$. Then there is an
	$\varepsilon\in[0,\frac{4}{5})$ such that to each $\alpha\in (\frac{1}{2}, 2)$ there corresponds a positive number $c=c(\varepsilon, \alpha) $ with the property
	\begin{equation}\label{251}
	\int_\Omega u^{\varepsilon-1}\Delta\ua\Delta udx\geq c\int_\Omega(\Delta u^{\frac{\alpha+\varepsilon}{2}})^2dx
	\end{equation}
	for all $u\in W_{\frac{\alpha+\varepsilon}{2}}$.
\end{cor}
\begin{proof}
	In this case, we have 
	\begin{equation}\label{dtm111}
	\beta=1,\ \ \ \gamma=\frac{\alpha+\varepsilon}{2}.
	\end{equation}	
	Thus $\alpha+\beta-2\gamma=1-\varepsilon$. Hence we need to show that there exists
	an $\varepsilon\in [0,\frac{4}{5})$ such that
	\begin{eqnarray}
	\gamma&>&\eta,\label{253}\\
	-(N-1)(\alpha+1)+\frac{3N(\alpha+\varepsilon)}{2}-(2+N)\eta&>&0,\label{dtm22}
	\end{eqnarray}
	where $\eta$ is defined by \eqref{22.38}.  Plugging \eqref{dtm111} into \eqref{22.38},  we derive
	\begin{equation}\label{255}
	\eta=\frac{-N\varepsilon^2+2(N+1+\alpha)\varepsilon+(N+2)\alpha^2-2(N+3)\alpha}{4(N-1)(1-\varepsilon)}.
	\end{equation}	
	Using this value of $\eta$ in \eqref{dtm22}, after some elementary calculations we arrive at
	\begin{eqnarray}
	\lefteqn{-4(N-1)^2+4(N+1)(N+2)\alpha-(N+2)^2\alpha^2}\nonumber\\
	&&>N(5N-8)\varepsilon^2+[2N(N+2)\alpha-4N(2N-5)]\varepsilon\equiv h(\varepsilon).\label{dtm33}
	\end{eqnarray}
	The right-hand side is a quadratic function in $\varepsilon$, which achieves its minimum value at
	\begin{equation}\label{22.51}
	\varepsilon=\frac{-(N+2)\alpha+2(2N-5)}{5N-8}.
	\end{equation} But this number is not always non-negative. 
	It becomes negative only when $\alpha>\frac{2(2N-5)}{N+2}$. Thus we take
	$$\varepsilon=\left\{\begin{array}{ll}
	0 & \mbox{if $ \alpha>\frac{2(2N-5)}{N+2}$},\\
	\frac{-(N+2)\alpha+2(2N-5)}{5N-8}& \mbox{otherwise.}
	\end{array}\right.$$
	Obviously, we have $\varepsilon\in[0,\frac{4-\alpha}{5})$.
	Next we will show that  $\varepsilon$ selected above satisfies \eqref{253}-\eqref{dtm22}.
	If $\varepsilon=0$, then
	\begin{equation}
	\eta=\frac{[(N+2)\alpha-2(N+3)]\alpha}{4(N-1)}<0
	\end{equation}
	for $\alpha<2$. Thus \eqref{253} is trivially true. Set $\varepsilon=0$ in \eqref{dtm33} to obtain 
	\begin{equation}
	-4(N-1)^2+4(N+1)(N+2)\alpha-(N+2)^2\alpha^2>0.
	\end{equation}
	Solutions to this inequality form the interval
	$$\left(\frac{2(N+1)-4\sqrt{N}}{N+2},\ \ \ \frac{2(N+1)+4\sqrt{N}}{N+2} \right),$$
	which contains the interval $(\frac{1}{2},2 )$ if $N\leq 4$. That is to say, if the space
	dimension does not exceed $4$, we can simply take $\varepsilon=0$. We will have to do a little bit more work if we want \eqref{251} to hold for all the space dimensions. To this end,
	we substitute \eqref{22.51} into
	\eqref{dtm33} to deduce
	$$-(N-2)\alpha^2+(3N-4)\alpha+2-N>0.$$
	Solutions to this inequality are the interval
	$$\left(\frac{3N-4-\sqrt{N(5N-8)}}{2(N-2)},\ \ \ \frac{3N-4+\sqrt{N(5N-8)}}{2(N-2)}\right),$$
	which contains the interval $(\frac{1}{2},2 )$ if $N>2$. To see \eqref{253}, we substitute \eqref{dtm111} and \eqref{255} into \eqref{253} to obtain
	$$(N-2)\varepsilon^2+(2N\alpha+4)\varepsilon+(N+2)\alpha^2-4(N+1)\alpha<0.$$
	Remember that $\varepsilon$ lies in the interval $(0, \frac{4-\alpha}{5})$ and the function on the left-hand side of the above inequality is an increasing function of $\varepsilon$ over the interval. Thus it is sufficient for us to prove$$
	H(\alpha)\equiv (N-2)\frac{(4-\alpha)^2}{25}+(2N\alpha+4)\frac{4-\alpha}{5}+(N+2)\alpha^2-4(N+1)\alpha<0.$$
	It is easy to see that $H(\alpha)$ is a convex quadratic function of $\alpha$. An elementary calculation shows that
	$$H(\frac{1}{2})<0,\ \ \ H(2)<0.$$
	Thus $H(\alpha)<0$ for each $\alpha\in(\frac{1}{2},2 ) $. The proof is complete.
	
\end{proof}
From our proof we see that this lemma can hold for more general $\alpha$.

Similarly, we can investigate the functional
$$J(u)=\int_\Omega u^{2\gamma-\alpha}\Delta\ln u\Delta\ua dx.$$
A simple calculation shows
\begin{equation}
\Delta\ln u=-\frac{1}{\gamma}u^{-2\gamma}|\nur|^2+\frac{1}{\gamma}u^{-\gamma}\Delta\ur.
\end{equation}
Plug this and \eqref{22.19} into $J(u)$ to obtain
\begin{equation}
\frac{\gamma^2}{\alpha}J(u)=\int_\Omega(\Delta\ur)^2dx+\frac{4(\alpha-2\gamma)}{\gamma}\int_\Omega|\nurt|^2\Delta\ur dx-\frac{16(\alpha-\gamma)}{\gamma}\int_\Omega|\nurt|^4dx.
\end{equation}
It is interesting to note that the arguments of Lemmas \ref{l2.4} and \ref{l2.5} do not work here. If 
\begin{equation}\label{dsm222}
\alpha-2\gamma>0,
\end{equation}
we can still mimic the proof of Lemma \ref{l2.6} to obtain the following lemma.
\begin{lem}\label{l2.7}
	Let $\Omega$ be a bounded convex domain in $\mathbb{R}^N$ and \eqref{dsm222} be satisfied. Set
	\begin{equation}
	\eta=\frac{(2+N)(\gamma-\alpha)\gamma+(\alpha-2\gamma)\alpha}{(N-1)(2\gamma-\alpha)}.
	\end{equation}
	If $\eta$ satisfies the inequalities
	\begin{eqnarray}
	\eta-\gamma &<&0 \ \  \ \mbox{and}\\
	(2+N)(\eta-\gamma)+(N-1)(\alpha-2\gamma)&<&0,
	\end{eqnarray}
	then there is a positive number $c=c(\alpha,\gamma, N)$ such that
	\begin{equation}
	J(u)\geq c\int_\Omega(\Delta\ur)^2dx.
	\end{equation}
\end{lem}

Finally, we remark that it is possible to extend the inequality \eqref{dsmm1} to other types of 
domains $\Omega$. For example, if the boundary of $\Omega$ is $C^2$, $\beta=1$, and $ \alpha=\gamma=\frac{1}{2}$, a result of \cite{X3} asserts that
\begin{equation} \int_\Omega\Delta u\frac{\Delta\sn}{\sn} dx\geq
c_0\left(\int_\Omega|\nabla^2\sn|^2
dx+\int_\Omega\frac{1}{u}|\nabla\sn|^4dx\right)-c_1\int_\Omega udx
\end{equation} 
for $u\in W_{\frac{1}{2}}$.
Here the complication is largely due to the fact that \eqref{n231} is no longer true in this case. In its place, we have 
\begin{equation}
\int_\Omega(\Delta u)^2dx+\int_\Omega|\nabla
u|^2dx\geq
c\int_\Omega|\nabla^2u|^2
dx.
\end{equation}
It is also interesting to pursue the case where the Neumann boundary condition is replaced with the Dirichlet boundary condition.
\section{The Approximate Problem}\label{sec3}
In this section we will show how to construct a sequence of positive, smooth approximate solutions. Then we proceed to derive a priori estimates for the sequence that hold under more 
general conditions than these in Theorems 1.1 and 1.2. Our approximation scheme is based upon the following lemma. 
\begin{lem} \label{l3.1}Let $\Omega$ be a bounded domain in $\mathbb{R}^N$ with
	Lipschitz boundary $\po$. Assume that $\alpha\geq 1$, $\varepsilon\in[0,1)$, $ n\in\mathbb{R}$, and
	\begin{equation}\label{om1}
	p>\max\left\{\frac{N}{2},\ \ 2
	\right\}.
	\end{equation} 
	Then for each $1>\tau>0$ and each $f\in \lno$
	there is a solution $(\rho, F)$ with $\rho\geq 0$ in the space $\left(W^{1,2}(\Omega)\cap L^\infty(\Omega)\right)^2$ to the
	problem
	\begin{eqnarray}
	-\mbd\left[\rnt\nabla F\right]+\tau F&=&\frac{\rho-f}{\tau} \ \ \ \mbox{in}\ \ \ \Omega,\label{291}\\
	- \Delta \rho^\alpha+\tau \rho^p &=&-\rt F 
	+\tau\ \ \ \mbox{in}\ \ \ \Omega,\label{or1}\\
	\nabla \rho^\alpha\cdot\nu&=&\nabla F\cdot\nu =0\ \ \
	\mbox{on} \ \ \ \po.\label{292}
	\end{eqnarray}
	Furthermore, we have that $\rho, F\in
	C^{0,\beta}(\overline{\Omega)}$ for some $\beta\in (0,1)$ and
	$\rho\geq c_0$ in $\Omega$ for some $c_0>0$,
	where $\beta, c_0$ depend on the given data.
\end{lem}

Of course, the equations \eqref{291}-\eqref{292} are satisfied in
the sense of distributions. The last term $\tau$ in \eqref{or1} has
been added to ensure that $\rho$ cannot be identically $0$. As we shall see, it is also
the main reason why $\rho$ has a positive lower bound. This idea was first employed in \cite{X2}. The real tricky part, though, is that we have used the term $\frac{\rho^{1-\varepsilon}}{\rho^{\alpha-\varepsilon}+\tau}$ to approximate $\rho^{1-\alpha}$. That is, a term with a negative exponent is being approximated by a term with two positive exponents. It serves two purposes: one is that we avoid having to seek solutions in a function space whose functions must have positive lower bounds; the other is that it ensures that solutions to \eqref{or1} is non-negative. If our solution is non-negative then the term $\tau$ in \eqref{or1} guarantees that it is  bounded away from zero below. If we further assume that $f$ is H\"{o}lder continuous on
$\overline{\Omega}$, then the classical Schauder theory \cite{GT} indicates that the pair $(\rho, F)$ is a classical solution. 
This, together with the fact that $\rho$ is bounded away from $0$ below, enables us to achieve higher regularity, thereby
justifying all our calculations in the derivation of a prior estimates for the sequence of approximate solutions to be constructed later. 

\begin{proof}
	We just need to modify the proof of Lemma \ref{l3.1} in \cite{X4}. 
	We still apply the Leray-Schauder
	Fixed Point Theorem
	(see Theorem 11.3 in \cite{GT}).
	For this purpose, we define an operator $B$ from $\lno$ into $\lno$
	as follows. Given that $\rho\in\lno$, we
	consider the problem
	\begin{eqnarray}
	-\textup{div}\left[\rpnt\nabla F\right]+\tau F=\frac{\rho-f}{\tau} \ \ \ &\mbox{in}&\ \ \ \Omega,\label{om2}\\
	\nabla F\cdot\nu =0\ \ \ &\mbox{on}& \ \ \ \po.
	\end{eqnarray}
	Eqn \eqref{om2}	is uniformly elliptic, and thus by \eqref{om1}
	we can appeal to the results in (\cite{GT}, Chap. 8) and thereby conclude that this linear boundary value problem has a unique solution $F$ in
	the space $ W^{1,2}(\Omega)\cap L^{\infty}(\Omega)$. 
	For each $q\geq2$, the function $|F|^{q-2}F\in  W^{1,2}(\Omega) $ and $\nabla\left(|F|^{q-2}F\right)=(q-1)|F|^{q-2}\nabla F$. Upon using it as a test function
	in \eqref{om2},
	we arrive at
	\begin{equation}\label{33.7}
	\|F\|_q\leq \frac{1}{\tau^2}\left\|\rho-f\right\|_q.
	\end{equation}
	
	Now we use the function $F$ so-obtained to form the problem
	\begin{eqnarray}
	-\Delta \psi+\tau|\psi|^{\frac{p}{\alpha}-1}\psi &=&-\rtp F 
	+\tau\ \ \ \mbox{in}\ \ \ \Omega,\label{or6}\\
	\nabla \psi\cdot\nu&=&0\ \ \ \mbox{on} \ \ \ \po.
	\end{eqnarray}
	Obviously, this problem has a unique solution $\psi$ in the space $
	W^{1,2}(\Omega)\cap L^\infty(\Omega)$.
	We define $$B(\rho)=\theta(\psi), \ \ \ \mbox{where
		$\theta(s)=|s|^{\frac{1}{\alpha}-1}s$.}$$ It is easy to see that $B:
	\lno\rightarrow\lno$ is well-defined. By Theorem 8.22 in \cite{GT} and a boundary flattening argument
	\cite{X5} , we can conclude that there exists a number $\beta\in
	(0,1)$, depending only on the given data,  such that
	$F, \psi\in C^{0,\beta}(\overline{\Omega})$.
	It is not difficult to show that the H\"{o}lder continuity of $\psi$ implies 
	that $B$ is continuous and maps bounded sets into
	precompact ones. 
	
	Next, we show that
	\begin{equation}\label{or7}
	\|\rho\|_{\infty}\leq c
	\end{equation}
	for all $\sigma\in [0,1]$ and $\rho$ such that $\sigma
	B(\rho)=\rho$. Here and in the remaining proof, $c$ is a generic positive number which depends only on the given data. Without loss of generality, assume $\sigma>0$.   Then
	the equation $\sigma B(\rho)=\rho$ is equivalent to the problem
	\begin{eqnarray}
	\frac{\rho-f}{\tau}&=& 	-\textup{div}\left[\rpnt\nabla F\right]+\tau F\ \ \ \mbox{in}\ \ \ \Omega,
	\label{mf1}
	\\
	- \Delta\theta^{-1} (\frac{\rho}{\sigma}) +\tau\left|\theta^{-1}
	(\frac{\rho}{\sigma})\right|^{\frac{p}{\alpha}-1}\theta^{-1}
	(\frac{\rho}{\sigma})
	&	=&-\rtp F	+\tau\ \ \ \mbox{in}\ \ \ \Omega,\label{om8}\\
	\nabla \theta^{-1} (\frac{\rho}{\sigma})\cdot\nu&=&\nabla
	F\cdot\nu =0\ \ \ \mbox{on}\ \ \ \po.
	\end{eqnarray}
	Remember that $\varepsilon <1$, and thus
	$\left(\theta^{-1}
	(\frac{\rho}{\sigma})\right)^{-}(\rho^+)^{1-\varepsilon}=0$ on $\Omega$.
	Upon using $\left(\theta^{-1} (\frac{\rho}{\sigma})\right)^{-}$ as a
	test function in \eqref{om8}, we deduce that  $ \rho\geq 0$ in
	$\Omega$. Subsequently, we have $$\theta^{-1}
	(\frac{\rho}{\sigma})=\frac{\rho^\alpha}{\sigma^\alpha}.$$ We can
	rewrite \eqref{om8} as
	\begin{equation}\label{om9}
	- \frac{1}{\sigma^\alpha}\Delta
	\rho^\alpha+\frac{\tau}{\sigma^p}\rho^p=-\rt F	+\tau\ \ \
	\mbox{in}\ \ \ \Omega.
	\end{equation}
	Integrate this equation to obtain
	\begin{eqnarray}
	\tau\int_{\Omega}\rho^p dx
	&=&-\sigma^p\int_{\Omega} F\rt dx+\tau|\Omega|\nonumber\\
	&\leq& \frac{1}{\tau}\left(\|F\|_{\frac{p}{p+\varepsilon-1}}+c\right)\|\rho\|_p^{1-\varepsilon}+c\nonumber\\
	&\leq&c\|\rho\|_p^{2-\varepsilon}+c\|\rho\|_p^{1-\varepsilon}+c.
	\end{eqnarray}
	The last step is due to the fact that $\frac{p}{p+\varepsilon-1}\leq p$. A simple application of the
	interpolation inequality
	$$ab\leq \eta a^p+c(\eta)b^q,\ \ \
	\frac{1}{p}+\frac{1}{q}=1$$ gives $$\|\rho\|_p\leq c.$$
	In the sequel, we will not acknowledge this interpolation inequality
	again when it is being used.
	
	Obviously, $\rt\leq \frac{1}{\tau}\rho^{1-\varepsilon}$. Applying the proof of Theorem 8.15 in (\cite{GT}, p.189), we
	can derive from \eqref{mf1} and \eqref{om9} that
	\begin{eqnarray} 
	\|F\|_\infty&\leq&
	c\left\|\frac{\rho-f}{\tau}\right\|_p\leq c,\\
	\|\rho^\alpha\|_\infty&\leq&
	c\|\rho^\alpha\|_2+c\|\rho^{1-\varepsilon}\|_p\leq c.\label{mf2}
	\end{eqnarray}
	Note that the constant $c$ here depends on $\tau$, but not the upper bound of the elliptic coefficient $(\rho+\tau)^n$ in \eqref{mf1}.
	This completes the proof of existence.
	
	Next, we show
	\begin{equation}\label{n92}
	\frac{1}{\rho}\in L^s(\Omega)\ \ \ \mbox{for each $s\geq 1$.}
	\end{equation}
	To this end, we
	use $\frac{1}{(\rho+\delta)^s}$, where $\delta>0
	$, as a test function in \eqref{or1} to obtain
	$$-s\alpha\io \frac{\rho^{\alpha-1}|\nabla\rho|^2}{(\rho+\delta)^{s+1}}dx+\tau\int_\Omega\frac{\rho^p}{(\rho+\delta)^s}dx =-\io\frac{\rho^{1-\varepsilon}F}{(\rho^{\alpha-\varepsilon}+\tau)(\rho+\delta)^s}dx+\tau\int_\Omega\frac{1}{(\rho+\delta)^s}dx.
	$$ Drop the first term and take a note of the fact that
	$$
	\left|\io\frac{\rho^{1-\varepsilon}F}{(\rho^{\alpha-\varepsilon}+\tau)(\rho+\delta)^s}dx\right|\leq \frac{1}{\tau}\io\frac{\rho^{1-\varepsilon}|F|}{(\rho+\delta)^s}dx\leq c\io (\rho+\delta)^{1-\varepsilon-s}dx
	$$ to derive
	$$
	\tau\int_\Omega\frac{1}{(\rho+\delta)^s}dx\leq \tau\int_\Omega(\rho+\delta)^{p-s}dx+c\io (\rho+\delta)^{1-\varepsilon-s}dx.
	$$
	Recall the interpretation inequality
	$$
	c\io (\rho+\delta)^{1-\varepsilon-s}dx=c\io\left(\frac{1}{\rho+\delta}\right)^{s-(1-\varepsilon)}dx\leq \frac{\tau}{2}\io\left(\frac{1}{\rho+\delta}\right)^{s}dx+c
	$$ and thereby obtain
	$$\int_\Omega\frac{1}{(\rho+\delta)^s}dx \leq
	c\int_\Omega (\rho+\delta)^{p-s}dx+c.$$ If $s\leq p$, then we take
	$\delta \rightarrow 0$ in the above inequality to obtain
	$$\int_\Omega\frac{1}{\rho^s}dx \leq
	c\int_\Omega \rho^{p-s}dx+c.$$ It is not difficult to see that this
	inequalities actually holds for each $s>1$, and thus \eqref{n92}
	follows.
	
	Now we let $v=\frac{1}{\rho^\alpha+\delta}, \delta>0$. Then we can
	easily show that $v$ satisfies the boundary value problem
	\begin{eqnarray*}
		-\Delta v+\frac{2}{v}|\nabla v|^2
		&=&\left(\rt F-\tau+\tau\rho^p\right)
		v^2\equiv G
		\ \ \ \mbox{in}\ \ \ \Omega,\\
		\nabla v\cdot\nu &=& 0\ \ \ \mbox{on}\ \ \ \partial\Omega
	\end{eqnarray*}
	in the sense of distributions. 
	We can conclude from \cite{GT,X4} again that
	$$\|v\|_\infty\leq c\|v\|_2+c\|G\|_p\leq c.$$ The last step is due to \eqref{n92}. This completes
	the proof of Lemma \ref{l3.1}.
\end{proof}

If $\alpha<1$, then our approximate problem can be made a little simpler. For the purpose of comparison, we state the corresponding result in the following
\begin{lem} \label{l3.2}Let $\Omega$ be a bounded domain in $\mathbb{R}^N$ with
	Lipschitz boundary $\po$. Assume that $\alpha\in (0,1)$, $ n\in\mathbb{R}$, and \begin{equation}\label{oom1}
	p>\max\{\frac{N}{2},\ \ 2
	\}.
	\end{equation} Then for each $1>\tau>0$ and each $f\in \lno$
	there is a solution $(\rho, F)$ with $\rho\geq 0$ in the space $\left(W^{1,2}(\Omega)\cap L^\infty(\Omega)\right)^2$ to the
	problem
	\begin{eqnarray}
	-\mbd\left[\rnt\nabla F\right]+\tau F&=&\frac{\rho-f}{\tau} \ \ \ \mbox{in}\ \ \ \Omega,\label{2291}\\
	- \Delta \rho^\alpha+\tau \rho^p &=&-F\rho^{1-\alpha}\nonumber\\
	&&+\tau\ \ \ \mbox{in}\ \ \ \Omega,\label{oor1}\\
	\nabla \rho^\alpha\cdot\nu&=&\nabla F\cdot\nu =0\ \ \
	\mbox{on} \ \ \ \po.\label{2292}
	\end{eqnarray}
	Furthermore, we have that $\rho, F\in
	C^{0,\beta}(\overline{\Omega)}$ for some $\beta\in (0,1)$ and
	$\rho\geq c_0$ in $\Omega$ for some $c_0>0$,
	where $\beta, c_0$ depend on the given data.
\end{lem}

The proof is similar to that of the previous lemma.

We are ready to construct our approximate solutions. Let $T>0$ be given. We divide the time interval $[0,T]$ into $j$
equal subintervals, $j\in\{1,2,\cdots\}$. Set
$$\tau=\frac{T}{j}.$$
We discretize and regularize the system \eqref{311}-\eqref{312} as
follows. For $k=1, \cdots, j$, solve recursively the systems
\begin{eqnarray}
\frac{\rho_k-\rho_{k-1}}{\tau}&=&-\textup{div}\left[\rnk\nabla \fk\right]+\tau\fk \ \ \ \mbox{in}\ \ \ \Omega,\label{293}\\
- \Delta \rk^\alpha+\tau \rk^p& =&-\rtk F_k
+\tau\ \ \ \mbox{in}\ \ \ \Omega,\label{294}\\
\nabla \rk^\alpha\cdot\nu&=&\nabla\fk\cdot\nu =0\ \ \
\mbox{on} \ \ \ \po,\\
\rho_0(x)&=&u_0(x). \label{295}
\end{eqnarray}
Define the functions
\begin{eqnarray*}
	\njt(x,t) &=& (t-t_{k-1})\frac{\rho_k(x)-\rho_{k-1}(x)}{\tau} +\rho_{k-1}(x), \ \ \  x\in\Omega, \ \  t\in (t_{k-1}, t_k],\\
	\nj(x,t) &=& \rks(x),\ \ \  x\in\Omega, \ \  t\in (t_{k-1}, t_k],\\
	\fj(x,t) &=& \fk(x),\ \ \  x\in\Omega, \ \  t\in (t_{k-1}, t_k],
\end{eqnarray*}
We can rewrite the system \eqref{293}-\eqref{295}
as
\begin{eqnarray}
\frac{\partial \njt}{\partial t}&=& -\textup{div}\left[\rnj\nabla\fj\right]+\tau\fj \ \ \mbox{in}\ \ \Omega_T,\label{n82}\\
- \Delta\nj^\alpha+\tau\nj^p &=&-\rtj\fj
+\tau\ \ \mbox{in}\ \ \Omega_T,\label{n91}\\
\nabla \nj^\alpha\cdot\nu&=&\nabla\fj\cdot\nu =0\ \ \ \mbox{on} \ \ \ \pot,\\
\overline{u}_j(x,0)&=&u_0(x) \ \ \ \mbox{on}\ \ \
\Omega.\label{n121}
\end{eqnarray}


\begin{lem} \label{l3.3}Let $\varepsilon\in [0,\frac{4}{5})$ be given as in Corollary \ref{cor2.2}. Assume that $\alpha\in [1, \frac{3}{2}), n\in (0,2-\varepsilon)$, and $p>\max\{\frac{N}{2},2\}$. Then there is a $\tau_0\in(0,1)$ such that
	\begin{eqnarray}
	\lefteqn{\int_{\Omega_t}(\Delta\nj^\alpha)^2
		dxds+\tau\int_{\Omega_t}|\Delta\nj^{\frac{\alpha+\varepsilon}{2}}|^2dxds}\nonumber\\
	&&+\tau\int_{\Omega_t}\nj^{p+\alpha-2}|\nabla\nj|^2dxds
	+\tau^2\int_{\Omega_t}\nj^{p+\varepsilon-2}|\nabla\nj|^2dxds\nonumber\\
	&&+\tau\int_{\Omega_t}\nj^{\alpha-2}|\nabla\nj|^2dxds+\tau^2\int_{\Omega_t}\nj^{\varepsilon-2}|\nabla\nj|^2dxds\nonumber\\
	&&+\max_{0\leq t\leq T}\int_\Omega G(\nj(x,t))dx\leq c\label{n58}
	\end{eqnarray}
	for all $\tau\in (0, \tau_0)$, where
	\begin{equation}
	G(s)=\left\{\begin{array}{ll}
	s & \mbox{if $n> 1$,}\\
	s^{2-n} & \mbox{if $n< 1$,}\\
	s\ln s-s & \mbox{if $n= 1$.}
	\end{array}\right.
	\end{equation}
	Here and in what follows $c$ denotes a positive
	constant independent of $j$.
\end{lem}

By the proof of  Corollary \ref{cor2.2}, we can take $\varepsilon=0$ if $N\leq 4$. Thus in this case $n\in (0,2)$.
\begin{proof} For $r\in[0,\infty)$ we define
	\begin{equation}\label{n56}
	K(r)=\int_1^{r}\frac{1}{(s+\tau)^n}ds=\left\{\begin{array}{ll}
	\frac{1}{1-n}\left[(r+\tau)^{1-n}-(1+\tau)^{1-n}\right]&\mbox{if $n\ne 1$,}\\
	\ln(s+\tau)-\ln(1+\tau)&\mbox{if $n= 1$.}
	\end{array}\right.
	\end{equation}
	We use $\kr$ as a test function in \eqref{293} to
	obtain
	\begin{equation}\label{314}
	\int_{\Omega}\fk\Delta\rk dx-\tau\int_{\Omega}\fk\kr
	dx+\frac{1}{\tau}\int_{\Omega}(\rk-\rho_{k-1})\kr dx
	=0.
	\end{equation} We proceed to estimate each integral in the above
	equation. For this purpose, we solve \eqref{294} for $\fk$
	to yield
	\begin{equation}\label{dsa11}
	\fk=\rk^{\alpha-1}\Delta\rk^\alpha+\tau\rkm\Delta\rk^\alpha-\tau\rk^{p+\alpha-1}-\tau^2\rk^{p+\varepsilon-1}+\tau\rk^{\alpha-1}+\tau^2\rkm.
	\end{equation}
	This can be done because $\rk$ is bounded away from $0$ below.
	Observe
	\begin{equation}
	\frac{1}{\tau}\int_{\Omega}(\rk-\rho_{k-1})\kr dx\geq\frac{1}{\tau}\int_\Omega\int_{\rho_{k-1}}^{\rk}K(r)drdx.
	\end{equation}
	This is due to the fact that $K(r)$ is an increasing function on $[0,\infty)$. Substituting \eqref{dsa11} into the first integral in \eqref{314} gives
	\begin{eqnarray}
	\int_{\Omega}\fk\Delta\rk dx
	&=&\int_{\Omega}\Delta\rk\rk^{\alpha-1}\Delta\rk^\alpha dx+\tau\int_{\Omega}\Delta\rk\rkm\Delta\rk^\alpha dx\nonumber\\
	&&+
	(p+\alpha-1)\tau\int_{\Omega}\rk^{p+\alpha-2}|\nabla\rk|^2dx+(p+\varepsilon-1)\tau^2\int_{\Omega}\rk^{p+\varepsilon-2}|\nabla\rk|^2dx\nonumber\\
	&&-(\alpha-1)\tau\int_{\Omega}\rk^{\alpha-2}|\nabla\rk|^2dx-(\varepsilon-1)\tau^2\int_{\Omega}\rk^{\varepsilon-2}|\nabla\rk|^2dx.\label{dsa22}
	\end{eqnarray}
	By Corollaries \ref{cor2.1} and \ref{cor2.2}, we have
	\begin{eqnarray}
	\int_\Omega\rkam\drka\drk&\geq&c\int_\Omega(\drka)^2dx,\label{dsa33}\\
	\int_\Omega\rkm\drka\drk&\geq&c\int_\Omega(\drk^{\frac{\alpha+\varepsilon}{2}})^2dx.\label{dsa44}
	\end{eqnarray}
	If $\alpha> 1$, then the coefficient of the sixth integral in \eqref{dsa22} is negative. To address this issue, we compute the integral as follows:
	\begin{eqnarray}
	\int_\Omega\rk^{\alpha-2}\nrks dx&=&\int_\Omega|\rk^{\frac{\alpha}{2}-1}\nrk|^2dx\nonumber\\
	&=&\frac{4}{\alpha^2}\int_\Omega|\nrk^{\frac{\alpha}{2}}|^2dx\nonumber\\
	&\leq& \frac{\delta}{\tau}\int_\Omega|\nrk^{\frac{\alpha}{2}}|^4dx+\tau c(\delta)\nonumber\\
	&\leq& \frac{9\delta}{16\tau}\int_\Omega|\drka|^2dx+c(\delta),\label{dsa66}
	\end{eqnarray}
	where $\delta$ is a positive number. Using \eqref{dsa33}-\eqref{dsa66} in
	\eqref{dsa22} and choosing $\delta$ suitably small, we obtain
	\begin{eqnarray}
	\int_{\Omega}\fk\Delta\rk dx
	&\geq&c\int_{\Omega}(\Delta\rk^\alpha)^2 dx+c\tau\int_{\Omega}(\Delta\rk^{\frac{\alpha+\varepsilon}{2}})^2 dx\nonumber\\
	&&+
	(p+\alpha-1)\tau\int_{\Omega}\rk^{p+\alpha-2}|\nabla\rk|^2dx+(p+\varepsilon-1)\tau^2\int_{\Omega}\rk^{p+\varepsilon-2}|\nabla\rk|^2dx\nonumber\\
	&&+(1-\varepsilon)\tau^2\int_{\Omega}\rk^{\varepsilon-2}|\nabla\rk|^2dx -c.\label{dsa77}
	\end{eqnarray}
	Plugging \eqref{dsa11} into the second integral in \eqref{314} yields
	\begin{eqnarray}
	-\tau\int_\Omega\fk\kr dx &=&-\tau\int_\Omega(\rkam+\tau\rkm)\kr\drka dx\nonumber\\
	&&+\tau^2\int_\Omega(\rk^{p+\alpha-1}+\tau\rk^{p+\varepsilon-1}-\rkam-\tau\rkm)\kr dx\nonumber\\
	&\equiv&I_{1,k}+I_{2,k}.
	\end{eqnarray}
	A simple integration by parts enables us to represent $I_{1,k}$ in the form
	\begin{eqnarray}
	I_{1,k}&=&\alpha\tau\int_\Omega((\alpha-1)\rk^{\alpha-\varepsilon}-(1-\varepsilon)\tau)\kr\rk^{\alpha+\varepsilon-3}\nrks dx\nonumber\\
	&&+\alpha\tau\int_\Omega\frac{\rk^{\alpha-\varepsilon}+\tau}{(\rk+\tau)^n}\rk^{\alpha+\varepsilon-2}\nrks dx.\label{33.44}
	\end{eqnarray}
	We first consider the case where
	\begin{equation}
	\alpha>1.
	\end{equation} Set
	\begin{equation}
	b_\tau\equiv\left(\frac{1-\varepsilon}{\alpha-1}\right)^{\frac{1}{\alpha-\varepsilon}}\tau^{\frac{1}{\alpha-\varepsilon}}.
	\end{equation}
	Then we can choose $\tau_0\in(0,1)$ so that 
	\begin{equation}
	b_{\tau_0}<1.
	\end{equation}
	From here on, we assume that
	\begin{equation}
	\tau\leq\tau_0.
	\end{equation}
	Recall from the definition of $K(r)$ that
	\begin{equation}
	K(r)(r-1)\geq 0 \ \ \ \mbox{on $[0,\infty)$}.\label{3.444}
	\end{equation}
	We can easily deduce that the integrand of the first integral in \eqref{33.44} is non-positive only
	on the set
	$$A_k=\{x\in\Omega: b_\tau\leq\rk(x)\leq 1\}.$$
	On this set, we have
	$$-\kr=\int_{\rk}^{1} \frac{1}{(s+\tau)^n}ds\leq \int_{\rk}^{1}\frac{1}{s^n}d\leq \left\{\begin{array}{ll}
	\frac{\rk^{1-n}}{n-1}& \mbox{if $n>1$},\\
	-\ln\rk & \mbox{if $n=1$},\\
	\frac{1}{1-n}& \mbox{if $n<1$.}
	\end{array}\right.$$
	Our assumptions on $\alpha,n,\varepsilon$ imply that
	\begin{eqnarray}
	\rk^{2\alpha-n-\varepsilon}&\leq& 1\ \ \ \mbox{on $A_k$, and }\\
	\tau\rk^{-(\alpha-\varepsilon)}&\leq& c\ \ \ \mbox{on $A_k$}
	\end{eqnarray}  Keeping these in mind, we calculate, for $n>1$, that
	\begin{eqnarray}
	I_{1,k}&\geq&\alpha\tau\int_{A_k}((\alpha-1)\rk^{\alpha-\varepsilon}-(1-\varepsilon)\tau)\kr\rk^{\alpha+\varepsilon-3}\nrks dx\nonumber\\
	&\geq&\alpha(\alpha-1)\tau\int_{A_k}\kr\rk^{2\alpha-3}\nrks dx\nonumber\\
	&\geq&-c\tau\int_{A_k}\rk^{2\alpha-2-n}\nrks dx\nonumber\\
	&\geq&-c\int_{A_k}\rk^{3\alpha-2-n-\varepsilon}\nrks dx\nonumber\\
	&\geq&-c\int_{A_k}\rk^{2\alpha-n-\varepsilon}|\nrk^{\frac{\alpha}{2}}|^2 dx\nonumber\\
	&\geq&- \delta\int_\Omega |\nrk^{\frac{\alpha}{2}}|^4 dx-c(\delta)\nonumber\\
	&\geq&- \delta\int_\Omega(\Delta\rka)^2 dx-c(\delta)
	,\label{352}
	\end{eqnarray}
	where $\delta>0$. The above inequality still holds if $n\leq 1$. Thus if $\delta$ is sufficiently small, this term can be incorporated into the second integral in \eqref{dsa77}.


	If $\alpha=1$, then we can express $I_{1,k}$ in the form
	\begin{equation}
	I_{1,k}=\tau\int_\Omega\left[-(1-\varepsilon)\tau\kr+\frac{\rk(\rk^{1-\varepsilon}+\tau)}{(\rk+\tau)^n}\right]\rk^{\varepsilon-2}\nrks dx.\label{33.48}
	\end{equation}
	Set
	$$B_k=\{x\in\Omega:\rk(x)\geq 1\}.$$
	On the set $B_k$, we have
	$$
	\kr\leq\left\{\begin{array}{ll}
	\frac{1}{n-1}& \mbox{if $n>1$,}\\
	\ln\rk& \mbox{if $n=1$,}\\
	\frac{1}{1-n}\rk^{1-n}& \mbox{if $n<1$.}
	\end{array}\right.$$
	Furthermore, there holds
	$$\rk^{\varepsilon-n-1}\leq\rk^{p-1}\ \ \ \mbox{on $B_k$.}$$
	For $n<1$, we estimate
	\begin{eqnarray}
	I_{1,k}&\geq&-(1-\varepsilon)\tau^2\int_\Omega\kr\rk^{\varepsilon-2}\nrks dx\nonumber\\
	&\geq&-c\tau^2\int_{B_k}\rk^{\varepsilon-1-n}\nrks dx\nonumber\\
	&\geq&-c\tau^2\int_{\Omega}\rk^{p-1}\nrks dx.\label{355}
	\end{eqnarray}
	In view of the coefficient of the fourth integral in \eqref{dsa77},  we just need to impose a further condition
	\begin{equation}
	c\tau_0<p,
	\end{equation}
	where $c$ is the same as the one in the last line of \eqref{355}. Then the fourth term in \eqref{dsa77} can absorb the term on the right-hand side of \eqref{355}. The case where $n\geq1$ can be handled in a similar manner.
	
	We can express $I_{2,k}$ in the form
	\begin{equation}
	I_{2,k}=\tau^2\int_\Omega\kr\rkm(\rk^{\alpha-\varepsilon}+\tau)(\rk^p-1)dx.
	\end{equation}
	The integrand in the above integral is always non-negative.
	
	Summarizing our preceding estimates, we obtain
	\begin{eqnarray}
	\lefteqn{\int_{\Omega}(\Delta\rk^\alpha)^2 dx+\tau\int_{\Omega}(\Delta\rk^{\frac{\alpha}{2}})^2 dx}\nonumber\\
	&&+
	\tau\int_{\Omega}\rk^{p+\alpha-2}|\nabla\rk|^2dx+\tau^2\int_{\Omega}\rk^{p+\varepsilon-2}|\nabla\rk|^2dx\nonumber\\
	&&+\tau\int_{\Omega}\rk^{\alpha-2}|\nabla\rk|^2dx+\tau^2\int_{\Omega}\rk^{\varepsilon-2}|\nabla\rk|^2dx +\frac{1}{\tau}\int_\Omega\int_{\rho_{k-1}}^{\rk}K(r)drdx\leq c\label{dmm11}
	\end{eqnarray}
	for $\tau\in(0,\tau_0)$. Multiplying through this inequality by $\tau$ and summing up over
	$k$, we obtain
	\begin{eqnarray}
	\lefteqn{\int_{\Omega_t}(\Delta\nj^\alpha)^2 dxds+\tau\int_{\Omega_t}(\Delta\nj^{\frac{\alpha}{2}})^2 dxds}\nonumber\\
	&&+
	\tau\int_{\Omega_t}\nj^{p+\alpha-2}|\nabla\nj|^2dxds+\tau^2\int_{\Omega_t}\nj^{p+\varepsilon-2}|\nabla\nj|^2dxds\nonumber\\
	&&+\tau\int_{\Omega_t}\nj^{\alpha-2}|\nabla\nj|^2dxds+\tau^2\int_{\Omega_t}\nj^{\varepsilon-2}|\nabla\nj|^2dxds +\int_\Omega\int_{u_0}^{\nj}K(r)drdx\leq c
	\end{eqnarray}
	for $\tau\in(0,\tau_0)$.
	By the definition of $K(r) $, we have 
	\begin{eqnarray}
	\int_{u_0}^{\nj}K(r)dr&=&\int_{u_0}^{1}K(r)dr+\int_{1}^{\nj}K(r)dr\nonumber\\
	& \geq &\left\{ \begin{array}{ll}
	\frac{(\nj+\tau)^{2-n}}{(1-n)(2-n)}-\frac{(1+\tau)^{1-n}\nj}{1-n}-c &\mbox{if $n\ne 1$,}\\
	(\nj+\tau)\ln(\nj+\tau)-(1+\ln(1+\tau))\nj -c &\mbox{if $n= 1$.}
	\end{array}\right.\label{1359}
	\end{eqnarray} Here the fact that the second integral in \eqref{1359} is bounded is due to our assumptions on
	$u_0$. The rest is rather obvious.
	The proof is complete.
\end{proof}	
\begin{lem}\label{l3.4} Let the assumptions of Lemma \ref{l3.3} hold.
	Then we have
	\begin{eqnarray}
	&&\int_{\Omega_t}\nj^{\alpha-1}(\Delta\nja)^2 dxds+\tau\int_{\Omega_t}\nj^{\varepsilon-1}(\Delta\nja)^2 dxds\nonumber\\
	&&+
	\tau\int_{\Omega_t}\nj^{p+2\alpha-3}|\nabla\rk|^2dxds+\tau^2\int_{\Omega_t}\nj^{p+\varepsilon+\alpha-3}|\nabla\nj|^2dxds\nonumber\\
	&&+\tau^2\int_{\Omega_t}\nj^{\alpha+\varepsilon-3}|\nabla\nj|^2dxds+\max_{0\leq t\leq T} \int_\Omega\nj^{1+(\alpha-n)^+}(x,t)dx\leq c.\label{dsa7711}
	\end{eqnarray}
\end{lem}

\begin{proof} 
	Here we use a different test function. Let 
	\begin{equation}\label{1n56}
	L(r)=\int_1^{r}\frac{\alpha s^{\alpha-1}}{(s+\tau)^n}ds.
	\end{equation}
	Then use $\lr$ as a test function in \eqref{293} to
	obtain
	\begin{equation}\label{3141}
	-\int_{\Omega}\nabla\fk\cdot\nabla\rka dx-\tau\int_{\Omega}\fk\lr
	dx+\frac{1}{\tau}\int_{\Omega}(\rk-\rho_{k-1})\lr dx
	=0.
	\end{equation}
	The first integral in the above equation is equal to  
	\begin{eqnarray}
	\int_{\Omega}\fk\Delta\rka dx
	&=&\int_{\Omega}\rk^{\alpha-1}(\Delta\rk^\alpha)^2 dx+\tau\int_{\Omega}\rk^{\varepsilon-1}(\Delta\rk^\alpha)^2 dx\nonumber\\
	&&+
	(p+\alpha-1)\alpha\tau\int_{\Omega}\rk^{p+2\alpha-3}|\nabla\rk|^2dx\nonumber\\
	&&+(p+\varepsilon-1)\alpha\tau^2\int_{\Omega}\rk^{p+\varepsilon+\alpha-3}|\nabla\rk|^2dx\nonumber\\
	&&-(\alpha-1)\alpha\tau\int_{\Omega}\rk^{2\alpha-3}|\nabla\rk|^2dx-(\varepsilon-1)\alpha\tau^2\int_{\Omega}\rk^{\alpha+\varepsilon-3}|\nabla\rk|^2dx.\label{dsa221}
	\end{eqnarray}
	Owing to  Lemma \ref{l2.4}, for each $\alpha \in [1,\frac{5}{3})$ there is a positive number $c$ with the property
	\begin{equation}
	\int_\Omega\rkam(\drka)^2dx\geq c\int_\Omega(\drk^{\frac{3\alpha-1}{2}})^2dx.\label{dsa331}
	\end{equation}
	If $\alpha> 1$, then the coefficient of the sixth integral in \eqref{dsa221} is negative. We will use \eqref{dsa331} to deal with the term. To do this, we estimate
	\begin{eqnarray}
	\int_\Omega\rk^{2\alpha-3}\nrks dx&=&\int_\Omega\rk^{\frac{\alpha-1}{2}}\rk^{\frac{3\alpha-5}{2}}|\nrk|^2dx\nonumber\\
	&\leq& \frac{\delta}{\tau}\int_\Omega\rk^{3\alpha-5}|\nrk|^4dx+\tau c(\delta)\int_\Omega \rk^{\alpha-1}\nonumber\\
	&=& \frac{4^4\delta}{(3\alpha-1)^4\tau}\int_\Omega|\nrk^{\frac{3\alpha-1}{4}}|^4dx+\tau c(\delta)\int_\Omega \rk^{\alpha-1}dx\nonumber\\
	&\leq& \frac{144\delta}{(3\alpha-1)^4\tau}\int_\Omega|\drk^{\frac{3\alpha-1}{2}}|^2dx+c(\delta)\int_\Omega \rk^{\alpha-1}dx\nonumber\\
	&\leq& \frac{c\delta}{\tau} \int_\Omega\rkam(\drka)^2dx+c(\delta)\int_\Omega \rk^{\alpha-1}dx,\label{dsa661}
	\end{eqnarray}
	where $\delta$ is a positive number. Using \eqref{dsa331}-\eqref{dsa661} in
	\eqref{dsa221} and choosing $\delta$ suitably small, we obtain
	\begin{eqnarray}
	\int_{\Omega}\fk\Delta\rka dx
	&\geq&c\int_\Omega\rkam(\drka)^2dx+c\tau\int_{\Omega}\rk^{\varepsilon-1}(\Delta\rka)^2 dx\nonumber\\
	&&+
	(p+\alpha-1)\alpha\tau\int_{\Omega}\rk^{p+2\alpha-3}|\nabla\rk|^2dx\nonumber\\
	&&+(p+\varepsilon-1)\alpha\tau^2\int_{\Omega}\rk^{p+\varepsilon+\alpha-3}|\nabla\rk|^2dx\nonumber\\
	&&+(1-\varepsilon)\alpha\tau^2\int_{\Omega}\rk^{\alpha+\varepsilon-3}|\nabla\rk|^2dx -c\int_\Omega\rk^{\alpha-1}.\label{dsa771}
	\end{eqnarray}
	Plugging \eqref{dsa11} into the second integral in \eqref{3141} yields
	\begin{eqnarray}
	-\tau\int_\Omega\fk\lr dx &=&-\tau\int_\Omega(\rkam+\tau\rkm)\lr\drka dx\nonumber\\
	&&+\tau^2\int_\Omega(\rk^{p+\alpha-1}+\tau\rk^{p+\varepsilon-1}-\rkam-\tau\rkm)\lr dx\nonumber\\
	&\equiv&J_{1,k}+J_{2,k}.
	\end{eqnarray}
	The term $J_{1,k}$ can be written in the form
	\begin{eqnarray}
	J_{1,k}&=&\alpha\tau\int_\Omega((\alpha-1)\rk^{\alpha-\varepsilon}-(1-\varepsilon)\tau)\lr\rk^{\alpha+\varepsilon-3}\nrks dx\nonumber\\
	&&+\alpha^2\tau\int_\Omega\frac{\rk^{\alpha-\varepsilon}+\tau}{(\rk+\tau)^n}\rk^{2\alpha+\varepsilon-3}\nrks dx.\label{33.441}
	\end{eqnarray}
	If	$\alpha>1$,
	we can define $A_k, b_\tau$ as before.
	Note that the integrand of the first integral in \eqref{33.441} is non-positive only
	on the set
	$A_k.$
	For $x\in A_k$,  we have
	\begin{equation}
	-\lr=\int_{\rk}^{1} \frac{\alpha s^{\alpha-1}}{(s+\tau)^n}ds\leq\int_{\rk}^{1}\alpha s^{\alpha-n-1}ds\leq\left\{\begin{array}{ll}
	\frac{\alpha}{\alpha-n}& \mbox{if $\alpha>n$,}\\
	-\alpha\ln\rk& \mbox{if $\alpha=n$,}\\
	\frac{\alpha}{\alpha-n}\rk^{\alpha-n}& \mbox{if $\alpha<n$.}
	\end{array}\right.
	\end{equation}
	If $\alpha<n$,  we have
	\begin{eqnarray}
	J_{1,k}&\geq&\alpha\tau\int_{A_k}((\alpha-1)\rk^{\alpha-\varepsilon}-(1-\varepsilon)\tau)\lr\rk^{\alpha+\varepsilon-3}\nrks dx\nonumber\\
	&\geq&\alpha(\alpha-1)\tau\int_{A_k}\lr\rk^{2\alpha-3}\nrks dx\nonumber\\
	&\geq&-c\tau\int_{A_k}\rk^{3\alpha-3-n}\nrks dx\nonumber\\
	&\geq&-c\int_{A_k}\rk^{4\alpha-3-n-\varepsilon}\nrks dx\nonumber\\
	&\geq&-c\int_{A_k}\rk^{\frac{5\alpha-1}{2}-n-\varepsilon}|\nrk^{\frac{3\alpha-1}{4}}|^2 dx\nonumber\\
	&\geq&- \delta\int_\Omega |\nrk^{\frac{3\alpha-1}{4}}|^4 dx-c(\delta)\nonumber\\
	&\geq&- \delta\int_\Omega(\Delta\rk^{\frac{3\alpha-1}{2}})^2 dx-c(\delta)\nonumber\\
	&\geq&- \delta\int_\Omega\rkam(\Delta\rka)^2 dx-c(\delta)
	,\label{1352}
	\end{eqnarray}
	where $\delta>0$. Thus $J_{1,k}$ can be absorbed into the second integral in \eqref{dsa771} if $\delta$ is small. If $\alpha\geq n$, a similar argument can be made.
	

	If $\alpha=1$, then we can express $J_{1,k}$ in the form
	\begin{equation}
	J_{1,k}=\tau\int_\Omega\left[-(1-\varepsilon)\tau\lr+\frac{\rk(\rk^{1-\varepsilon}+\tau)}{(\rk+\tau)^n}\right]\rk^{\varepsilon-2}\nrks dx.\label{133.48}
	\end{equation}
	Let
	$B_k=\{x\in\Omega:\rk(x)\geq 1\}$ be given as before.
	On the set $B_k$, we have
	$$
	\lr\leq\left\{\begin{array}{ll}
	\frac{1}{n-1}& \mbox{if $n>1$,}\\
	\ln\rk& \mbox{if $n=1$,}\\
	\frac{1}{1-n}\rk^{1-n}& \mbox{if $n<1$.}
	\end{array}\right.$$
	Furthermore, there holds
	$$\rk^{\varepsilon-n-1}\leq\rk^{p-1}\ \ \ \mbox{on $B_k$.}$$
	For $n<1$, we estimate
	\begin{eqnarray}
	J_{1,k}&\geq&-(1-\varepsilon)\tau^2\int_\Omega\lr\rk^{\varepsilon-2}\nrks dx\nonumber\\
	&\geq&-c\tau^2\int_{B_k}\rk^{\varepsilon-1-n}\nrks dx\nonumber\\
	&\geq&-c\tau^2\int_{\Omega}\rk^{p-1}\nrks dx.\label{3551}
	\end{eqnarray}
	In view of the coefficient of the fourth integral in \eqref{dsa771},  we just need to impose a further condition
	\begin{equation}
	c\tau_0<p,
	\end{equation}
	where $c$ is the same as the one in the last line of \eqref{3551}. The case where $n\geq 1$ can be handled in a similar manner.
	
	We can express $J_{2,k}$ in the form
	\begin{equation}
	J_{2,k}=\tau^2\int_\Omega\lr\rkm(\rk^{\alpha-\varepsilon}+\tau)(\rk^p-1)dx.
	\end{equation}
	The integrand in the above integral is always non-negative.
	
	If $n>1$ and $ \alpha\ne n$, we have
	\begin{eqnarray}
	L(r)&=&\int_{1}^{r}\frac{\alpha s^{\alpha-1}}{(s+\tau)^n}ds\nonumber\\
	&=&\int_{1}^{r}\alpha s^{\alpha-1}d\frac{(s+\tau)^{1-n}}{1-n}\nonumber\\
	&=&\frac{\alpha(1+\tau)^{1-n}}{n-1}-\frac{\alpha}{n-1}r^{\alpha-1}(r+\tau)^{1-n}+\frac{\alpha(\alpha-1)}{n-1}\int_{1}^{r}\frac{s^{\alpha-2}}{(s+\tau)^{n-1}}ds\nonumber\\
	&\geq&\left\{\begin{array}{ll}
	\frac{\alpha}{\alpha-n}(r+\tau)^{\alpha-n}+\frac{\alpha(1+\tau)^{1-n}(\alpha-n)-\alpha(\alpha-1)(1+\tau)^{\alpha-n}}{(n-1)(\alpha-n)} & \mbox{if $r>1$,}\\
	\frac{\alpha}{\alpha-n}r^{\alpha-n}+\frac{\alpha(1+\tau)^{1-n}(\alpha-n)-\alpha(\alpha-1)}{(n-1)(\alpha-n)} & \mbox{if $r\leq1$.}	
	\end{array}\right.\label{3375}
	\end{eqnarray}
	Similarly, if $n>1$ and $ \alpha= n$, we have
	\begin{equation}
	L(r)
	\geq\left\{\begin{array}{ll}
	n \ln\frac{r+\tau}{1+\tau}+\frac{n((1+\tau)^{1-n}-1)}{n-1}& \mbox{if $r>1$,}\\
	n\ln r+\frac{n((1+\tau)^{1-n}-1)}{n-1} & \mbox{if $r\leq1$.}	
	\end{array}\right.
	\end{equation}
	Thus we always have
	$$\int_{\Omega}\int^{\overline{u}_j}_{u_0}L(r)dr
	dx\geq c\int_{\Omega} \nj^{1+(\alpha-n)^+} dx-c,$$ where
	$\overline{u}_j=\overline{u}_j(x,t)$, provided that $n>1$. It is not difficult to see the same inequality holds for $n\leq 1$. 
	Collecting all the previous estimates in \eqref{3141}, we arrive at
	\begin{eqnarray}
	&&\int_{\Omega}(\Delta\rk^{\frac{3\alpha-1}{2}})^2 dx+\tau\int_{\Omega}\rk^{\varepsilon-1}(\Delta\rka)^2 dx\nonumber\\
	&&+
	\tau\int_{\Omega}\rk^{p+2\alpha-3}|\nabla\rk|^2dx+\tau^2\int_{\Omega}\rk^{p+\varepsilon+\alpha-3}|\nabla\rk|^2dx\nonumber\\
	&&+\tau^2\int_{\Omega}\rk^{\alpha+\varepsilon-3}|\nabla\rk|^2dx +\frac{1}{\tau}\int_\Omega\int_{\rho_{k-1}}^{\rk}L(r)drdx\nonumber\\
	&\leq &c\int_\Omega\rk^{\alpha-1} dx+c.
	\end{eqnarray}
	Multiply through the inequality by $\tau$, note that $0\leq\alpha-1<1$, and sum up over $k$ to obtain the desired result.
	The
	proof is complete.
\end{proof}
\begin{lem}\label{l3.5}Let the assumptions of Lemma \ref{l3.3} hold. Then the sequence $\{\nja\}$ is bounded in
	$L^2(0,T;W^{2,2}(\Omega))$.
\end{lem}
\begin{proof}
	Note that
	$$\nabla\nj=\frac{2}{\alpha}\nj^{\frac{2-\alpha}{2}}\nabla\nj^{\frac{\alpha}{2}}.$$
	We calculate
	\begin{eqnarray}
	\int_\Omega|\nabla\nj|^{\frac{\alpha N}{\alpha+N}} dx &=& \left(\frac{2}{\alpha}\right)^{\frac{\alpha N}{\alpha+N}}\int_\Omega\nj^{\frac{(2-\alpha)\alpha N}{2(\alpha+N)}}|\nabla\nj^{\frac{\alpha}{2}}|^{\frac{\alpha N}{\alpha+N}}dx\nonumber\\
	&\leq & c\left(\int_\Omega\nj^{\frac{2(2-\alpha)\alpha N }{4(\alpha+N)-\alpha N}}\right)^{1-\frac{\alpha N}{4(\alpha+N)}}\left(\int_\Omega |\nabla\nj^{\frac{\alpha}{2}}|^4dx\right)^{\frac{\alpha N}{4(\alpha+N)}}\nonumber\\
	&\leq & c\left(\int_\Omega |\Delta\nj^\alpha|^2dx\right)^{\frac{\alpha N}{4(\alpha+N)}}.
	\end{eqnarray}
	The last step is due to the fact that
	$$\frac{2(2-\alpha)\alpha N }{4(\alpha+N)-\alpha N}\leq 1.$$
	On account of the Sobolev Embedding Theorem, we have
	\begin{eqnarray}
	\left(\int_\Omega\nja dx\right)^{\frac{1}{\alpha}}&\leq& c\left(\int_\Omega|\nabla\nj|^{\frac{\alpha N}{\alpha+N}} dx\right)^{\frac{\alpha+N}{\alpha N}}+c\int_\Omega\nj dx\nonumber\\
	&\leq&c\left(\int_\Omega |\Delta\nj^\alpha|^2dx\right)^{\frac{1}{4}}+c.
	\end{eqnarray}
	Consequently, there holds
	\begin{equation}
	\int_{0}^{T}\left(\int_{\Omega}\nja dx\right)^{\frac{4}{\alpha}}dt\leq c.
	\end{equation}
	Recall the interpolation inequality
	$$	\int_\Omega|\nabla\nja|^2dx\leq c\int_\Omega|\nabla^2\nja|^2dx+c\left(\int_\Omega\nja dx\right)^2.$$
	This, together with the fact that $\alpha\in [1, \frac{3}{2})$, implies the desired result. 		
	
\end{proof}
\begin{lem}\label{l3.6}Let the assumptions of Lemma \ref{l3.3} hold. Then we have $$\tau\int_{\Omega_T}\nj^{p+2\alpha-1+((\alpha-n)^++1)\frac{2}{N}}dxdt\leq c.$$
\end{lem}
\begin{proof} By the Sobolev inequality, we estimate, for $\alpha>n$, that
	\begin{eqnarray*}
		\int_{\Omega_T}\nj^{p+2\alpha-1+(\alpha-n+1)\frac{2}{N}} dx dt&\leq&
		c\int^T_0\left(\int_\Omega\nj^{\frac{p+2\alpha-1}{2}
			\frac{2N}{N-2}}dx\right)^{\frac{N-2}{N}}\left(\int_\Omega\nj^{\alpha-n+1}
		dx\right)^{\frac{2}{N}}dt\\
		&\leq &
		c\left(\int_{\Omega_T}|\nabla\nj^{\frac{p+2\alpha-1}{2}}|^2dxdt+\int_{\Omega_T}\nj^{p+2\alpha-1}dxdt\right)\nonumber\\
		&&\cdot\left(\sup_{0\leq
			t\leq T}\int_\Omega\nj ^{\alpha-n+1}dx\right)^{\frac{2}{N}}\\
		&\leq
		&c\left(\int_{\Omega_T}\nj^{p+2\alpha-3}|\nabla\nj|^2dxdt+\int_{\Omega_T}\nj^{p+2\alpha-1}dxdt\right)\\
		&\leq &
		c\int_{\Omega_T}\nj^{p+2\alpha-3}|\nabla\nj|^2dxdt\nonumber\\
		&&+\delta\int_{\Omega_T}\nj^{p+2\alpha-1+(\alpha-n+1)\frac{2}{N}}dxdt+c.
	\end{eqnarray*}
	Choosing $\delta$ suitably small yields
	$$\int_{\Omega_T}\nj^{p+2\alpha-1+(\alpha-n+1)\frac{2}{N}}dxdt\leq
	c\int_{\Omega_T}\nj^{p+2\alpha-3}|\nabla\nj|^2dxdt+c.$$
	If $\alpha\leq n$, we have 
	\begin{equation}
	\int_{\Omega_T}\nj^{p+2\alpha-1+\frac{2}{N}}dxdt\leq
	c\int_{\Omega_T}\nj^{p+2\alpha-3}|\nabla\nj|^2dxdt+c.
	\end{equation}
	Multiplying through the inequality by $\tau$ and taking a note of
	Lemma \ref{l3.3} give the desired result.
\end{proof}
\section{Proof of Theorem \ref{thm1.1}}\label{sec4}
The proof is divided into several lemmas.
\begin{lem}\label{l4.1}Let the assumptions of Lemma \ref{l3.3} hold. If  $n\geq 1$,
	then $\tau\fj\rightarrow 0$ strongly in $L^1(\Omega_T)$.
\end{lem}
\begin{proof} 	Recall that
	\begin{equation}
	\tau\fj=\tau\nj^{\alpha-1}\Delta\nja+\tau^2\nj^{\varepsilon-1}\Delta\nja-
	\tau^2\nj^{p+\alpha-1}-\tau^3\nj^{p+\varepsilon-1}+
	\tau^2\nj^{\alpha-1}+\tau^3\nj^{\varepsilon-1}.
	\end{equation}
	We will show that each term on the right hand side of the above equation tends to $0$
	strongly in $L^1(\Omega_T)$ as $\tau\rightarrow 0$. We begin with the last term. For this purpose, assume $\tau\leq\tau_0$, where $\tau_0$ is given as in Lemma \ref{l3.3}.
	Set
	$$I_{2,j}=\tau^2\int_{\Omega_T} K(\nj)\nj^{\varepsilon-1}(\nj^{\alpha-\varepsilon}+\tau)(\nj^p-1)dxdt.$$
	By the proof of Lemma \ref{l3.3}, we have
	\begin{equation}\label{dtm11}
	I_{2,j}\leq c.
	\end{equation}
	Let
	$$A_j=\{(x,t)\in \Omega_T:\nj(x,t)\leq 1 \},\ \ \ B_j=\Omega_T\setminus A_j.$$
	Then we can rewrite \eqref{dtm11} as
	\begin{eqnarray}
	\lefteqn{\tau^2\int_{B_j}K(\nj)\nj^{p+\alpha-1}dxdt+\tau^3\int_{B_j}K(\nj)\nj^{p+\varepsilon-1}dxdt}\nonumber\\
	&&-\tau^2\int_{A_j}K(\nj)\nj^{\alpha-1}dxdt-\tau^3\int_{A_j}K(\nj)\nj^{\varepsilon-1}dxdt\nonumber\\
	&&\leq -\tau^2\int_{A_j}K(\nj)\nj^{p+\alpha-1}dxdt-\tau^3\int_{A_j}K(\nj)\nj^{p+\varepsilon-1}dxdt\nonumber\\
	&&+\tau^2\int_{B_j}K(\nj)\nj^{\alpha-1}dxdt+\tau^3\int_{B_j}K(\nj)\nj^{\varepsilon-1}dxdt+c.\label{33.57}
	\end{eqnarray}
	On the set $B_j$, we have
	\begin{equation}
	K(\nj)\leq\left\{\begin{array}{ll}
	\frac{1}{n-1} & \mbox{if $n>1$,}\\
	\ln\nj& \mbox{if $n=1$,}
	\end{array}\right.
	\end{equation}
	while on the set $A_j$, there holds
	\begin{equation}
	-K(\nj)\leq\left\{\begin{array}{ll}
	\frac{1}{n-1}\nj^{1-n} & \mbox{if $n>1$,}\\
	-\ln\nj& \mbox{if $n=1$.}
	\end{array}\right.
	\end{equation}
	We wish to show that the right-hand side of \eqref{33.57} is bounded. If $n>1$, we 
	have
	\begin{equation}
	-\tau^2\int_{A_j}K(\nj)\nj^{p+\alpha-1}dxdt\leq\frac{\tau^2}{n-1}\int_{A_j}\nj^{p+\alpha-n}dxdt\leq c\tau^2.
	\end{equation}
	The last step is due to the fact that $p+\alpha-n\geq 0$.
	The second integral on the right-hand side of \eqref{33.57} can be handled in an entirely similar way. The third one there can be estimated as follows:
	\begin{equation}
	\tau^2\int_{B_j}K(\nj)\nj^{\alpha-1}dxdt\leq c\tau^2\int_{B_j}\nj^{\alpha}dxdt\leq  c\tau^2 .
	\end{equation}
	Here we have used Lemma \ref{l3.5} and the fact that $\ln\nj\leq\nj$ on the set $B_j$.
	As for the last integral, remember that $\varepsilon-1<0$. Hence $\nj^{\varepsilon-1}\leq 1$ on $B_j$. Subsequently, we have
	\begin{equation}
	\tau^3\int_{B_j}K(\nj)\nj^{\varepsilon-1}dxdt\leq c\tau^3\int_{B_j}\nj dxdt\leq c\tau^3.
	\end{equation}
	Now we can conclude that 
	\begin{equation}
	-\tau^3\int_{A_j}K(\nj)\nj^{\varepsilon-1}dxdt\leq c.
	\end{equation}
	This implies
	\begin{equation}
	\tau^3\int_{\Omega_T}\nj^{\varepsilon-1}dxdt\rightarrow 0\ \ \ \mbox{as $\tau\rightarrow 0$.}
	\end{equation}
	To see this, we calculate
	\begin{eqnarray}
	\tau^3\int_{\Omega_T}\nj^{\varepsilon-1}dxdt&=&\tau^3\int_{\{\nj\leq\tau\}}\nj^{\varepsilon-1}dxdt+\tau^3\int_{\{\nj>\tau\}}\nj^{\varepsilon-1}dxdt\\
	&\leq&\frac{1}{|K(\tau)|}\tau^3\int_{\{\nj\leq\tau\}}|\kr|\nj^{\varepsilon-1}dxdt+c\tau^{2+\varepsilon}\\
	&\leq&\frac{c}{|K(\tau)|}+c\tau^{2+\varepsilon}\rightarrow 0 \ \ \ \mbox{as $\tau\rightarrow 0$.}
	\end{eqnarray}
	Our assumption that $n\geq 1$ is made just to ensure that $|K(\tau)|\rightarrow \infty$ as $\tau\rightarrow 0$.
	
	We can derive from Lemma \ref{l3.3}
	that
	\begin{eqnarray}
	\int_{\Omega_T}\tau\nj^{\alpha-1}|\Delta\nja|
	dxdt&\leq&\left(\int_{\Omega_T}\tau^2\nj^{2\alpha-2}dxdt\right)^{\frac{1}{2}}\left(\int_{\Omega_T}(\Delta\nja)^2dxdt\right)^{\frac{1}{2}}\nonumber\\
	&\leq& c\tau
	\end{eqnarray}
	because $2\alpha-2<1$.
	With the aid of Lemma \ref{l3.4}, we obtain
	\begin{eqnarray}
	\int_{\Omega_T}\tau^2\nj^{\varepsilon-1}|\Delta\nja|
	dxdt&\leq&\left(\int_{\Omega_T}\tau^3\nj^{\varepsilon-1}dxdt\right)^{\frac{1}{2}}\left(\tau\int_{\Omega_T}\nj^{\varepsilon-1}(\Delta\nja)^2dxdt\right)^{\frac{1}{2}}\nonumber\\
	&\leq
	&c\left(\int_{\Omega_T}\tau^3\nj^{\varepsilon-1}dxdt\right)^{\frac{1}{2}}\nonumber\\
	&\rightarrow& 0.
	\end{eqnarray}
	We deduce from Lemma \ref{l3.5} that
	\begin{eqnarray}\int_{\Omega_T}\tau^2\nj^{p+\alpha-1}dxdt&=&\int_{\{\nj\leq 1\}}\tau^2\nj^{p+\alpha-1}dxdt+\int_{\{\nj>1\}}\tau^2\nj^{p+\alpha-1}dxdt\nonumber\\
	&\leq
	&c\tau^{2}+\int_{\{\nj>1\}}\tau^2\nj^{p+2\alpha-1+((\alpha-n)^++1)\frac{2}{N}}dxdt\nonumber\\
	&\leq&c\tau^{2}+c\tau\rightarrow 0 \ \ \ \mbox{as $\tau\rightarrow
		0$.}
	\end{eqnarray}
	Similarly, we can show that $\tau^3\int_{\Omega_T}\nj^{p+\varepsilon-1}dxdt\rightarrow 0$ as $\varepsilon\rightarrow 0$. This completes the proof.
\end{proof}
\begin{lem}\label{l4.2}Let the assumptions of Lemma \ref{l4.1} hold. If $n\leq 1+\frac{\sigma}{4}$, then the sequence
	$\{\partial_t\njt\}$ is bounded in $L^1((0,T);
	(W^{2,\infty}(\Omega))^*)$, where $\sigma$ is given as in \eqref{1115}.
\end{lem}
\begin{proof}
	We first claim that
	\begin{equation}\label{399}
	\int_{\Omega_T}\nj^{\sigma+2\alpha}dxdt\leq c.
	\end{equation}
	This estimate is a consequence of Lemmas \ref{l3.3} and \ref{l3.5}. Indeed, 
	Lemma \ref{l3.3} says that $\nj(x,t)$ is bounded in $L^\infty(0,T; L^1(\Omega))$, while Lemma \ref{l3.5} asserts that $\nj^\alpha(x,t)$ is bounded in $ L^2(0,T; W^{2,2}(\Omega))$. If $N>4$, then we have from  \eqref{1115} that $\sigma=\frac{4}{N}<1$. We estimate from H\"{o}lder's inequality and the Sobolev Embedding Theorem that
	\begin{eqnarray}
	\int_{\Omega_T}\nj^{\frac{4}{N}+2\alpha}dxdt&=&\int_{0}^{T}\left(\int_{\Omega}\nj dx\right)^{\frac{4}{N}}\left(\int_{\Omega}\nj^{\alpha\frac{2N}{N-4}}dx\right)^{\frac{N-4}{N}}dt\nonumber\\
	&\leq &\left(\max_{0\leq t\leq T}\int_{\Omega}\nj dx\right)^{\frac{4}{N}}\int_{0}^{T}\left(\int_{\Omega}\nj^{\alpha\frac{2N}{N-4}}dx\right)^{\frac{N-4}{N}}dt\nonumber\\
	&\leq &c\int_{0}^{T}\|\nj^{\alpha}\|_{W^{2,2}(\Omega)}^2 dt\leq c.
	\end{eqnarray}
	If $N=4$, then $\sigma\in(0,1)$ according to \eqref{1115}. Subsequently,
	\begin{eqnarray}
	\int_{\Omega_T}\nj^{\sigma+2\alpha}dxdt
	&\leq &\left(\max_{0\leq t\leq T}\int_{\Omega}\nj dx\right)^{\sigma}\int_{0}^{T}\left(\int_{\Omega}\nj^{\alpha\frac{2}{1-\sigma}}dx\right)^{1-\sigma}dt\nonumber\\
	&\leq &c\int_{0}^{T}\|\nj^{\alpha}\|_{W^{2,\frac{2}{2-\sigma}}(\Omega)}^2 dt\leq c.
	\end{eqnarray}
	The last step is due to $\frac{2}{2-\sigma}<2$.
	If $N<4$, then $\sigma=1$ by \eqref{1115}. Consequently, we have
	\begin{eqnarray}
	\int_{\Omega_T}\nj^{1+2\alpha}dxdt&\leq&\int_{0}^{T}\int_\Omega\nj dx\|\nja\|^2_\infty dt\nonumber\\
	&\leq &c\int_{0}^{T}\left(\|\nabla^2\nja\|^2_2+\|\nja\|_2^2\right)dt\leq c.
	\end{eqnarray} 
	This completes the proof of \eqref{399}.
	
	Recall that
	\begin{eqnarray}
	(\nj+\tau)^{n-1}\nabla \nj	\fj&=&\nj^{\alpha-1}(\nj+\tau)^{n-1}\nabla \nj\Delta\nja+\tau\nj^{\varepsilon-1}(\nj+\tau)^{n-1}\nabla \nj\Delta\nja\nonumber\\
	&&-\tau\nj^{p+\alpha-1}(\nj+\tau)^{n-1}\nabla \nj-\tau^2\nj^{p+\varepsilon-1}(\nj+\tau)^{n-1}\nabla \nj\nonumber\\
	&&+\tau\nj^{\alpha-1}(\nj+\tau)^{n-1}\nabla \nj+\tau^2\nj^{\varepsilon-1}(\nj+\tau)^{n-1}\nabla \nj.\label{3101}
	\end{eqnarray}
	Our objective here is to show that each term on the right-hand side of the above equation is bounded in $(L^1(\Omega_T))^N$. To this end, we note 
	$$(\nj+\tau)^{n-1}\leq \nj^{n-1}+\tau^{n-1}$$ 
	since $n-1<1$. By our assumption, $0<-\alpha+4n-1\leq 2\alpha+\sigma$.
	We compute
	\begin{eqnarray}
	\int_{\Omega_T}\nj^{\alpha+n-2}|\nabla \nj\Delta\nja|dxdt &=&
	\frac{4}{3\alpha-1}\int_{\Omega_T}\nj^{-\frac{\alpha}{4}+n-\frac{1}{4}}|\nabla \nj^{\frac{3\alpha-1}{4}}|\nj^{\frac{\alpha-1}{2}}|\Delta\nja|dxdt\nonumber\\
	&\leq&c\left(\int_{\Omega_T}\nj^{-\alpha+4n-1}dxdt\right)^{\frac{1}{4}}\left(\int_{\Omega_T}|\nabla\nj^{\frac{3\alpha-1}{4}}|^4dxdt\right)^{\frac{1}{4}}\nonumber\\
	&&\cdot\left(\int_{\Omega_T}\nj^{\alpha-1}|\Delta\nja|^2dxdt\right)^{\frac{1}{2}} \nonumber\\
	&\leq& c.\label{4420}
	\end{eqnarray}
	There are too many terms on the right-hand side of \eqref{3101}, and so we will skip the obvious ones. 
	Now we look at the second term on the right-hand side of \eqref{3101}. We have
	\begin{eqnarray}
	\lefteqn{	\tau\int_{\Omega_T}\nj^{\varepsilon+n-2}|\nabla \nj\Delta\nja|dxdt}\nonumber\\
	&=&
	\frac{4}{\alpha+\varepsilon}\int_{\Omega_T}\tau^{\frac{1}{4}}\nj^{\frac{\varepsilon}{4}-\frac{\alpha}{4}+n-\frac{1}{2}}\tau^{\frac{1}{4}}|\nabla \nj^{\frac{\alpha+\varepsilon}{4}}|\tau^{\frac{1}{2}}\nj^{\frac{\varepsilon-1}{2}}|\Delta\nja|dxdt\nonumber\\
	&\leq&c\left(\int_{\Omega_T}\tau\nj^{\varepsilon-\alpha+4n-2}dxdt\right)^{\frac{1}{4}}\left(\int_{\Omega_T}\tau|\nabla\nj^{\frac{\alpha+\varepsilon}{4}}|^4dxdt\right)^{\frac{1}{4}}\nonumber\\
	&&\cdot\left(\int_{\Omega_T}\tau\nj^{\varepsilon-1}|\Delta\nja|^2dxdt\right)^{\frac{1}{2}} \nonumber\\
	&\leq& c\tau^{\frac{1}{4}}. \label{442}
	\end{eqnarray}	
	Here we have used the fact that $0<\varepsilon-\alpha+4n-2\leq 2\alpha+\sigma$.
	Next we estimate
	\begin{eqnarray}
	\tau\int_{\Omega_T}\nj^{p+\alpha+n-2}|\nabla\nj|dxdt&\leq&\left(\int_{\Omega_T}\tau\nj^{p+\alpha-2}|\nabla\nj|^2dxdt\right)^{\frac{1}{2}}\left(\int_{\Omega_T}\tau\nj^{p+\alpha-2+2n}dxdt\right)^{\frac{1}{2}}\nonumber\\
	&\leq&c.
	\end{eqnarray}
	The last step is due to Lemma \ref{l3.6} because $n\leq 1+\frac{\sigma}{4}$. The rest of the terms can be estimated similarly.
	
	We still
	need to consider the term
	\begin{eqnarray}
	(\nj+\tau)^n	\fj&=&\nj^{\alpha-1}(\nj+\tau)^n\Delta\nja+\tau\nj^{\varepsilon-1}(\nj+\tau)^n\Delta\nja\nonumber\\
	&&-\tau\nj^{p+\alpha-1}(\nj+\tau)^n-\tau^2\nj^{p+\varepsilon-1}(\nj+\tau)^n\nonumber\\
	&&+\tau\nj^{\alpha-1}(\nj+\tau)^n+\tau^2\nj^{\varepsilon-1}(\nj+\tau)^n.\label{4423}
	\end{eqnarray}
	It is easy to see that it is also bounded in
	$L^1(\Omega_T)$. Let $\xi$ be a $C^\infty$ test function with
	$\nabla\xi\cdot\nu=0$ on $\po$. We have
	\begin{eqnarray}
	(\partial_t\njt, \xi)&=&\int_\Omega(\nj+\tau)^n\nabla\fj\cdot\nabla\xi
	dx+\tau\int_\Omega\fj\xi dx\nonumber\\
	&=&-\int_\Omega\fj\left(n(\nj+\tau)^{n-1}\nabla\nj\cdot\nabla\xi+(\nj+\tau)^n\Delta\xi\right)
	dx+\tau\int_\Omega\fj\xi
	dx,\label{n134}
	\end{eqnarray}
	where $(\cdot,\cdot)$ is the duality pairing between
	$W^{2,\infty}(\Omega)$ and its dual space
	$(W^{2,\infty}(\Omega))^*$, from which the lemma follows.
\end{proof}

\begin{lem}\label{l4.3} Let the assumptions of Lemma \ref{l4.2} hold. Then the sequence
	$\{\nj\}$ is precompact in $L^{2\alpha}((0,T);L^{2\alpha}(\Omega))$.
\end{lem}
\begin{proof} Set $$q=\frac{8\alpha+4\sigma}{4+\sigma},$$
	where $\sigma$ is given as before.
	By our assumption on $\alpha$, we obviously have $q>2\alpha$.
	We estimate that
	\begin{eqnarray*}
		\int_{\Omega_T}|\nabla\nj|^q dxdt &=&
		\frac{2^q}{\alpha^q}\int_{\Omega_T}\nj^{\frac{(2-\alpha)q}{2}}|\nabla\nj^{\frac{\alpha}{2}}|^qdxdt\\
		&\leq
		&c\left(\int_{\Omega_T}|\nabla\nj^{\frac{\alpha}{2}}|^{4}dxdt\right)^{\frac{q}{4}}
		\left(\int_{\Omega_T}\nj^{\frac{2(2-\alpha)q}{4-q}}dxdt\right)^{1-\frac{q}{4}}.
	\end{eqnarray*}
	Note that $\frac{2(2-\alpha)q}{4-q}=2\alpha+\sigma$. Therefore, we obtain from \eqref{399}
	\begin{equation}
	\int_{\Omega_T}|\nabla\nj|^q dxdt\leq c.
	\end{equation}
	We can easily deduce from the definitions of $\nj, \njt$ that
	\begin{eqnarray}
	\int_{\Omega_T}|\njt|^{2\alpha}dxdt &\leq &	\int_{\Omega_T}|\nj|^{2\alpha}dxdt +\frac{1}{2}\tau\int_\Omega| u_0|^{2\alpha}dx, \\
	\int_{\Omega_T}|\nabla\njt|^{2\alpha}dxdt &\leq &	\int_{\Omega_T}|\nabla\nj|^{2\alpha}dxdt+\frac{1}{2}\tau\int_\Omega|\nabla u_0|^{2\alpha}dx.
	\end{eqnarray}
	Thus $\{\njt\}$ is bounded in $L^{2\alpha}((0,T);W^{1, 2\alpha}(\Omega))$.
	Note that for $t\in (t_{k-1},t_k]$ we have
	$$\njt(x,t)-\nj(x,t)=(t_k-t)\partial_t\njt(x,t).$$
	This together with Lemma \ref{l4.2} implies that
	\begin{equation}
	\int_0^T\|\nj-\njt\|_{(W^{2,\infty}(\Omega))^*}dt\leq
	c\tau.\label{n116}
	\end{equation}
	Observe that the embedding $W^{1,2\alpha}(\Omega)\hookrightarrow
	L^{2\alpha}(\Omega)$ is compact and
	$L^{2\alpha}(\Omega)\hookrightarrow\left(W^{2,\infty}(\Omega)\right)^*$ is
	continuous. A result of \cite{S} asserts that $\{\njt\}$ is
	precompact in both $L^{2\alpha}((0,T);L^{2\alpha}(\Omega))$ and $L^1((0,T);
	(W^{2,\infty}(\Omega))^*)$. According to \eqref{n116}, we also have
	that $\{\nj\}$ is precompact in $L^1((0,T);
	(W^{2,\infty}(\Omega))^*)$. This puts us in a position to apply the
	results in \cite{S} again, from which the lemmas follows. The proof
	is complete.
\end{proof}

We are ready to complete the proof of Theorem 1.1. We can extract a
subsequence of $\{j\}$, still denoted by $\{j\}$, such that
\begin{eqnarray}
\nj &\rightarrow& u\ \ \ \mbox{strongly in $L^{2\alpha}(\Omega_T)$ and
	a.e.},\label{3110}\\
\nja &\rightharpoonup& \ua\ \ \ \mbox{weakly in $L^2((0,T);
	W^{2,2}(\Omega))$.} \end{eqnarray} 
Equipped with this, we calculate
that
\begin{equation}\int_{\Omega_T}|\nabla\nja|^2dxdt=-\int_{\Omega_T}\Delta\nja\nja
dxdt\rightarrow-\int_{\Omega_T}\Delta\ua\ua
dxdt=\int_{\Omega_T}|\nabla\ua|^2dxdt.\end{equation} This implies
that
\begin{equation}
\nja \rightarrow \ua\ \ \ \mbox{strongly in $L^2((0,T);
	W^{1,2}(\Omega))$.}
\end{equation}
Without loss of generality, we may also assume
\begin{equation}
\nabla\nja \rightarrow \nabla\ua\ \ \ \mbox{a.e. on $\Omega_T$.}
\end{equation}
Note that $\alpha\geq 1$ and $\nabla\nj=\frac{1}{\alpha}\nj^{\alpha-1}\nabla\nja$. This along with \eqref{3110} shows
\begin{equation}
\nabla\nj \rightarrow \nabla u\ \ \ \mbox{a.e. on $\Omega_T$.}
\end{equation}

Next we wish to prove 
\begin{equation}\label{435n}
(\nj+\tau)^{n-1}\fj\nabla\nj\rightharpoonup \frac{2}{\alpha}u^{\frac{\alpha}{2}+n-1}\Delta \ua\nabla\uat\ \ \ \mbox{weakly in $L^1(\Omega_T)$.}
\end{equation}
This can be derived from the proof of Lemma \ref{l4.2}. To see this, first observe that
\begin{equation}
\nj^{\alpha+n-2}\nabla\nj \rightarrow u^{\alpha+n-2}\nabla u\ \ \ \mbox{a.e. on $\Omega_T$.}
\end{equation}
According to Egoroff's Theorem, to each $\delta>0$ there corresponds a set $E_\delta\subset\Omega_T$ with the property
\begin{equation}
\nj^{\alpha+n-2}\nabla\nj \rightarrow u^{\alpha+n-2}\nabla u\ \ \ \mbox{uniformly on  $\Omega_T\setminus E_\delta$ and $|E_\delta|<\delta$.}
\end{equation}
Due to our assumption, we have $\sigma-4n+4>0$. By a calculation identical to \eqref{4420}, we obtain
\begin{eqnarray}
\left|\int_{E_\delta}\nj^{\alpha+n-2}\nabla\nj\Delta\nja dxdt\right|&\leq&c\left(\int_{E_\delta}\nj^{2\alpha+4n-4}dxdt\right)^{\frac{1}{4}}\nonumber\\
&\leq&c\left(\int_{E_\delta}\nj^{2\alpha+\sigma}dxdt\right)^{\frac{\alpha+2n-2}{2(2\alpha+
		\sigma)}}|E_\delta|^{\frac{\sigma-4n+4}{4(2\alpha+\sigma)}}\nonumber\\
&\leq& c\delta^{\frac{\sigma-4n+4}{4(2\alpha+\sigma)}}.
\end{eqnarray}
Consequently, we have
\begin{eqnarray}
\lefteqn{\limsup_{j\rightarrow\infty}\left|\int_{\Omega_T}\nj^{\alpha+n-2}\nabla\nj\Delta\nja dxdt-\int_{\Omega_T}u^{\alpha+n-2}\nabla u\Delta \ua dxdt\right|}\nonumber\\
&&\leq c\delta^{\frac{\sigma-4n+4}{4(2\alpha+\sigma)}}+\left|\int_{E_\delta}\frac{2}{\alpha}u^{\frac{\alpha}{2}+n-1}\Delta \ua\nabla\uat dxdt\right|
\end{eqnarray}
The right-hand side goes to $0$ as $\delta\rightarrow 0$. Therefore, 
\begin{equation}
\nj^{\alpha+n-2}\nabla\nj\Delta\nja\rightharpoonup u^{\alpha+n-2}\nabla u\Delta \ua \ \ \ \mbox{weakly in $L^1(\Omega_T)$.}
\end{equation}
We can also prove
\begin{equation}
\int_{\Omega_T}\tau\nj^{p+\alpha-2+2n}dxdt\rightarrow 0\ \ \ \mbox{as $\tau\rightarrow 0$.}
\end{equation}
In this case, we use the inequality
$$\nj^{p+\alpha-2+2n}\leq \delta \nj^{p+\alpha+\frac{2}{N}}+c(\delta),\ \ \ \delta>0.$$
Then apply Lemma \ref{l3.6} to yield the desired result. The remaining terms on the right-hand side of \eqref{399} are very easy to handle. Thus \eqref{435n} follows.

On account of \eqref{4423}, we have
$$(\nj+\tau)^n\fj\rightharpoonup  u^{\alpha+n-1}\Delta \ua\ \ \ \mbox{weakly in $L^1(\Omega_T)$.}$$
We can infer from \eqref{n116} that
\begin{equation}
\njt\rightarrow u \ \ \ \mbox{strongly in $L^{2\alpha}(\Omega_T)$.}
\end{equation}
Assume $\xi(x,T)=0$ in \eqref{n134}, integrate it over $(0, T)$,
then let $j\rightarrow \infty$, and thereby obtain the theorem. The
proof is complete.
\section{Proof of Theorem \ref{thm1.2}}\label{sec5}

The proof of Theorem 1.2 relies on the following lemma
\begin{lem}\label{l5.1} Let the assumptions of Lemma \ref{l3.3} hold. Assume 
	\begin{equation}
	\alpha=1, \ \ \ \frac{1}{2}<\beta<n.
	\end{equation}
	Then there is $\tau_0\in (0,1)$ such that for all $\tau\in (0,\tau_0)$ we have
	\begin{eqnarray}
	&&\int_{\Omega_t}(\Delta\nj^{\frac{1+\beta}{2}})^2 dxds+\tau\int_{\Omega_t}(\Delta\nj^{\frac{\varepsilon+\beta}{2}})^2 dxds\nonumber\\
	&&+
	\tau\int_{\Omega_t}\nj^{p+\beta-2}|\nabla\rk|^2dxds+\tau^2\int_{\Omega_t}\nj^{p+\varepsilon+\beta-3}|\nabla\nj|^2dxds\nonumber\\
	&&+\tau^2\int_{\Omega_t}\nj^{\beta+\varepsilon-3}|\nabla\nj|^2dxds+\tau^2\int_{\Omega_t}(1+\tau\nj^{\varepsilon-1})(\nj^p-1)M(\nj)dxds\leq c,\label{dsa77111}
	\end{eqnarray}
	where
	\begin{equation}
	M(r)=\int_1^{r}\frac{\beta s^{\beta-1}}{(s+\tau)^n}ds.
	\end{equation}
\end{lem}

\begin{proof} Let $M(r)$ be given as above.
	We use $\mr$ as a test function in \eqref{293} to
	obtain
	\begin{equation}\label{31411}
	-\int_{\Omega}\nabla\fk\cdot\nabla\rkb dx-\tau\int_{\Omega}\fk\mr
	dx+\frac{1}{\tau}\int_{\Omega}(\rk-\rho_{k-1})\mr dx
	=0.
	\end{equation}
	The first integral in the above equation is equal to  
	\begin{eqnarray}
	\int_{\Omega}\fk\Delta\rkb dx
	&=&\int_{\Omega}\Delta\rk\Delta\rk^\beta dx+\tau\int_{\Omega}\rk^{\varepsilon-1}\Delta\rk\Delta\rkb dx\nonumber\\
	&&+
	p\beta\tau\int_{\Omega}\rk^{p+\beta-2}|\nabla\rk|^2dx\nonumber\\
	&&+(p+\varepsilon-1)\beta\tau^2\int_{\Omega}\rk^{p+\varepsilon+\beta-3}|\nabla\rk|^2dx\nonumber\\
	&&-(\varepsilon-1)\beta\tau^2\int_{\Omega}\rk^{\beta+\varepsilon-3}|\nabla\rk|^2dx.\label{dsa2211}
	\end{eqnarray}
	By virtue of Lemma \ref{l2.5}, we have
	\begin{equation}
	\int_\Omega\Delta\rk\Delta\rkb dx\geq c\int_\Omega(\drk^{\frac{\beta+1}{2}})^2dx, \ \ \ \ \beta \in(\frac{1}{2},2),\label{1dsa331}
	\end{equation}
	while  Corollary \ref{cor2.2} implies
	\begin{equation}
	\int_\Omega\rk^{\varepsilon-1}\Delta\rk\Delta\rkb dx\geq c\int_\Omega(\drk^{\frac{\beta+\varepsilon}{2}})^2dx, \ \ \ \ \beta \in(\frac{1}{2},2).\label{ddsa331}
	\end{equation}
	Using \eqref{1dsa331}-\eqref{ddsa331} in
	\eqref{dsa2211}, we obtain
	\begin{eqnarray}
	\int_{\Omega}\fk\Delta\rka dx
	&\geq&c\int_\Omega(\drk^{\frac{\beta+1}{2}})^2dx+c\tau\int_{\Omega}(\Delta\rk^{\frac{\beta+\varepsilon}{2}})^2 dx\nonumber\\
	&&+
	p\beta\tau\int_{\Omega}\rk^{p+\beta-2}|\nabla\rk|^2dx\nonumber\\
	&&+(p+\varepsilon-1)\beta\tau^2\int_{\Omega}\rk^{p+\varepsilon+\beta-3}|\nabla\rk|^2dx\nonumber\\
	&&-(\varepsilon-1)\beta\tau^2\int_{\Omega}\rk^{\beta+\varepsilon-3}|\nabla\rk|^2dx.
	\label{dsa771111}
	\end{eqnarray}
	We calculate the second integral in \eqref{31411} to obtain
	\begin{eqnarray}
	-\tau\int_\Omega\fk\mr dx &=&-\tau\int_\Omega(1+\tau\rkm)\mr\drk dx\nonumber\\
	&&+\tau^2\int_\Omega(\rk^{p}-1)(1+\tau\rkm)\mr dx\nonumber\\
	&\equiv&K_{1,k}+K_{2,k}.
	\end{eqnarray}
	Notice that $M(r)$ changes from negative to positive at $1$, and thus we always have
	\begin{equation}
	K_{2,k}\geq 0.
	\end{equation}
	The term $K_{1,k}$ can be written in the form
	\begin{eqnarray}
	K_{1,k}&=&-\tau^2\int_\Omega(1-\varepsilon)\mr\rk^{\varepsilon-2}\nrks dx\nonumber\\
	&&+\beta\tau\int_\Omega\frac{\rk^{1-\varepsilon}+\tau}{(\rk+\tau)^n}\rk^{\beta+\varepsilon-2}\nrks dx.\label{133.441}
	\end{eqnarray}
	Let
	$B_k=\{x\in\Omega:\rk(x)\geq 1\}$ be given as before.
	On the set $B_k$, we have
	$$
	\mr\leq
	\frac{\beta}{n-\beta}.$$
	Keeping this in mind, we estimate
	\begin{eqnarray}
	K_{1,k}&\geq&-(1-\varepsilon)\tau^2\int_\Omega\mr\rk^{\varepsilon-2}\nrks dx\nonumber\\
	&\geq&-c\tau^2\int_{B_k}\nrks dx\nonumber\\
	&\geq&-c\tau^2\int_{\Omega}\rk^{p+\beta-2}\nrks dx.\label{13551}
	\end{eqnarray}
	In view of the coefficient of the fourth integral in \eqref{dsa771111},  we just need to select a number $\tau_0$ in $(0,1)$ with the property
	\begin{equation}
	c\tau_0<p\beta,
	\end{equation}
	where $c$ is the same as the one in the last line of \eqref{13551}. Then $K_{1,k}$
	can be absorbed into the fourth term in \eqref{dsa771111}.
	
	By a calculation similar to \eqref{3375}, we have
	\begin{equation}
	M(r)
	\geq\left\{\begin{array}{ll}
	\frac{\beta}{\beta-n}\left[(r+\tau)^{\beta-n}-(1+\tau)^{\beta-n}\right] & \mbox{if $r>1$,}\\
	\frac{\beta}{\beta-n}r^{\beta-n} & \mbox{if $r\leq1$.}	
	\end{array}\right.
	\end{equation}
	Thus we always have
	$$\int_{\Omega}\int^{\overline{u}_j}_{u_0}M(r)dr
	dx\geq -c.$$
	The remaining proof is similar to that of Lemma \ref{l3.3}. 
	The
	proof is complete.
\end{proof}
We are ready to conclude the proof of Theorem 1.2. Since $n>\frac{1}{2}$, we can pick a number $\beta$ with the property
$$\frac{1}{2} <\beta<\min\{1,n\}.$$ Then we apply Lemma \ref{l5.1} to obtain
\begin{equation}
\tau^2\int_{\Omega_T}(1+\tau\nj^{\varepsilon-1})(\nj^p-1)M(\nj)dxds\leq c.
\end{equation}
This combined with the fact that
\begin{equation}
\lim_{r\rightarrow 0^+}M(r)=-\infty
\end{equation}
implies
\begin{equation}
\tau\fj\rightarrow 0 \ \ \ \mbox{strongly in $L^(\Omega_T)$.}
\end{equation}
We can easily infer this from the proof of Lemma \ref{l4.1}. That is, if we replace $K(r)$ with
$M(r)$ in the proof, all the arguments there still work.
By examining the rest of the calculations in the proof of Theorem 1.1, we see that all of them are still applicable here except  \eqref{442}, for which we make some adjustments. To this end,
we set $\alpha=1, \gamma=\frac{1+\beta}{2}, u=\nj$ in \eqref{22.19} to obtain
\begin{eqnarray}
\Delta\nj&=&\frac{2(1-\beta)}{(1+\beta)^2}\nj^{-\beta}|\nabla\nj^{\frac{1+\beta}{2}}|^2+\frac{2}{1+\beta}\nj^{1-\frac{1+\beta}{2}}\Delta\nj^{\frac{1+\beta}{2}}\nonumber\\
&=&\frac{8(1-\beta)}{(1+\beta)^2}\nj^{\frac{1-\beta}{2}}|\nabla\nj^{\frac{1+\beta}{4}}|^2+\frac{2}{1+\beta}\nj^{\frac{1-\beta}{2}}\Delta\nj^{\frac{1+\beta}{2}}.
\end{eqnarray}
Substitute this into the left-hand side of \eqref{442} to obtain
\begin{eqnarray}
\lefteqn{	\tau\int_{\Omega_T}\nj^{\varepsilon+n-2}|\nabla \nj\Delta\nj|dxdt}\nonumber\\
&\leq&\frac{8(1-\beta)}{(1+\beta)^2}\tau\int_{\Omega_T}\nj^{\frac{1-\beta}{2}+\varepsilon+n-2}|\nabla\nj||\nabla\nj^{\frac{1+\beta}{4}}|^2dxdt\nonumber\\
&&+\frac{2}{1+\beta}\tau\int_{\Omega_T}\nj^{\frac{1-\beta}{2}+\varepsilon+n-2}|\nabla\nj\Delta\nj^{\frac{1+\beta}{2}}|dxdt\nonumber\\
&\equiv&A_1+A_2.
\end{eqnarray}
We  estimate $A_2$ to yield
\begin{eqnarray}
A_2&\leq&c\left(\int_{\Omega_T}\tau^2\nj^{\beta+\varepsilon-3+2(n-\beta)+\varepsilon}|\nabla\nj|^2dxdt\right)^{\frac{1}{2}}\left(\int_{\Omega_T}|\Delta\nj^{\frac{1+\beta}{2}}|^2dxdt
\right)^{\frac{1}{2}} \nonumber\\
&\leq&c\left(\int_{\Omega_T}\tau^2\nj^{\beta+\varepsilon-3+2(n-\beta)+\varepsilon}|\nabla\nj|^2dxdt\right)^{\frac{1}{2}}. 
\end{eqnarray}
Thus if the exponent $\beta+\varepsilon-3+2(n-\beta)+\varepsilon<0$, then there holds
the inequality
\begin{eqnarray}
\nj^{\beta+\varepsilon-3+2(n-\beta)+\varepsilon}&=&\left(\frac{1}{\nj}\right)^{-(\beta+\varepsilon-3)-[2(n-\beta)+\varepsilon]}\nonumber\\
&\leq&\delta \left(\frac{1}{\nj}\right)^{-(\beta+\varepsilon-3)}+c(\delta)\nonumber\\
&=&\delta \nj^{\beta+\varepsilon-3}+c(\delta).
\end{eqnarray}
Consequently, we can deduce from Lemma \ref{l5.1} that
\begin{eqnarray}
\limsup_{\tau\rightarrow 0}A_2&\leq&\limsup_{\tau\rightarrow 0}c\left(\delta\int_{\Omega_T}\tau^2\nj^{\beta+\varepsilon-3}|\nabla\nj|^2dxdt+c(\delta)\tau^2\int_{\Omega_T}|\nabla\nj|^2dxdt\right)^{\frac{1}{2}}\nonumber\\
&\leq &c\delta^{\frac{1}{2}}.
\end{eqnarray}
Since $\delta$ is arbitrary, we have $\lim_{\tau\rightarrow 0}	A_2 =0$. 
If the exponent $\beta+\varepsilon-3+2(n-\beta)+\varepsilon\geq0$, then we use the
inequality
$$\nj^{\beta+\varepsilon-3+2(n-\beta)+\varepsilon}\leq \delta\nj^{p-1}+c(\delta).$$
This can be done because from our assumptions we always have $\beta+\varepsilon-3+2(n-\beta)+\varepsilon<p-1$. We can conclude from Lemma \ref{l3.3} that
$\lim_{\tau\rightarrow 0}	A_2 =0$. The term $A_1$ can be handled in the exactly same way.
This completes the proof.

\bigskip

\noindent{\bf Acknowledgment:} Portion of this work was completed while the second author
was visiting Duke University. He would like to express his gratitude for the hospitality of
the hosting institution and the financial support from the KI-Net for his visit. The research of JL was partially supported by
KI-Net NSF RNMS grant No. 1107291 and NSF grant DMS 1514826.

\end{document}